\documentclass[reqno,12pt]{amsart}
\usepackage{amsmath,amssymb,latexsym,soul,cite,mathrsfs}
\usepackage{blindtext}
\usepackage{enumitem}
\usepackage{ragged2e}
\usepackage[utf8]{inputenc}
\usepackage{amsthm,amsfonts}
\usepackage{comment}
\usepackage[mathcal]{eucal}
\usepackage{latexsym}
\usepackage[utf8]{inputenc}
\usepackage{bm}
\usepackage[all]{xy}
\usepackage{xcolor}
\usepackage{graphicx}
\usepackage{hyperref}

\usepackage[hyperpageref]{backref}

\newtheorem{theorem}{Theorem}

\newtheorem{theo}{Theorem}[section]

\newtheorem{remark}{Remark}[section]

\newtheorem{lemma}{Lemma}[section]

\newtheorem{prop}{Proposition}[section]

\newtheorem{cor}{Corollary}[section]

\newtheorem{conj}{Question}

\numberwithin{equation}{section}

\title[Sharp Sobolev and Adams-Trudinger-Moser inequalities]{Sharp Sobolev and Adams-Trudinger-Moser inequalities for symmetric functions without boundary conditions on hyperbolic spaces}

\author[J. M.\ do \'O]{Jo\~ao Marcos do \'O}
\author[G. Lu]{Guozhen Lu}
\author[R.~Ponciano]{Raon\'{\i} Ponciano}

\address[Jo\~{a}o Marcos do \'O]{Dep. Mathematics,
	Federal University of Para\'{\i}ba
	\newline\indent
	58051-900, Jo\~ao Pessoa-PB, Brazil}
\email{\href{mailto:jmbo@mat.ufpb.br}{jmbo@mat.ufpb.br}}

\address[Guozhen Lu]{Dep. Mathematics, University of Connecticut
	\newline\indent
	06269, Storrs-CT, United States of America}
\email{\href{mailto:guozhen.lu@uconn.edu}{guozhen.lu@uconn.edu}}

\address[Raon\'{\i}~Ponciano]{Ctr. Mathematics, Computing \& Cognition,
	Federal University of ABC
	\newline\indent
	09280-560, Santo Andr\'e, Brazil}
\email{\href{mailto:raoni.ponciano@ufabc.edu.br}{raoni.ponciano@ufabc.edu.br}}

\subjclass[2020]{Primary 26D10, 46E30, 46E35, 35J60; secondary 35J30, 58J70, 35A23}

\keywords{Hyperbolic Space; Radial Lemmata; Decay Lemmata; Adams-Trudinger-Moser inequalities; Weighted Sobolev spaces; radial functions without any boundary conditions. }

\begin{document}

\begin{abstract}
Embedding theorems for symmetric functions without any boundary conditions have been studied on flat Riemannian manifolds, such as the Euclidean space. However, these theorems have only been established on hyperbolic spaces for functions with homogeneous Dirichlet conditions. In this work, we focus on sharp Sobolev and Adams–Trudinger–Moser embeddings for radial functions in hyperbolic spaces, considering both bounded and unbounded domains. One of the main features of our approach is that we do not assume any boundary conditions for symmetric functions on geodesic balls or the entire hyperbolic space. Our main results include Theorems \ref{theo10}, \ref{theo11}, \ref{theo12}, which establish weighted Sobolev embedding theorems, and Theorems \ref{theo16} together with \ref{theo15}, which present Adams-Trudinger-Moser type of embedding theorems. In particular, a key result is Theorem \ref{theouv}, a highly nontrivial comparison between norms of the higher order covariant derivatives and higher order derivatives of the radial functions. Higher order asymptotic behavior of radial functions on hyperbolic spaces is established to prove our main theorems. This approach includes novel radial lemmata and decay properties of higher order radial Sobolev functions defined in hyperbolic space.
\end{abstract}

\maketitle

\begin{center}
		\footnotesize
		\tableofcontents
	\end{center}

\section{Introduction}
Sobolev embeddings are essential in analyzing partial differential equations (PDEs), as they establish connections between different function spaces. Their significance extends beyond pure mathematics, with applications in various fields such as geometry, physics, and engineering.

In particular, on flat Riemannian manifolds such as the Euclidean spaces, Sobolev and Adams-Trudinger-Moser embedding theorems for symmetric functions without any boundary conditions have been studied by numerous authors starting from the work of de Figueiredo et al. \cite{zbMATH05978504} and subsequent work by do \'{O}, Lu and Ponciano \cite{arXiv:2302.02262}. 
However, such embedding theorems have remained open for any curved manifolds. The authors of this work have established Sobolev embeddings for a class of weighted Sobolev spaces associated with a quasi-linear elliptic operator that generalizes the radial forms of the $k$-Hessian and $p$-Laplacian.

Hyperbolic spaces are complete noncompact Riemannian manifolds of negative constant curvature. This paper presents new Sobolev embeddings for radial functions defined on hyperbolic space, expanding the available tools for solving PDEs and addressing related problems without any boundary conditions.

We first recall embedding theorems on Euclidean balls for radial functions without any boundary conditions. Let $L^q_\theta(B_R^{\mathbb R})$ be the space of all measurable functions $u\colon B_R^{\mathbb R}\to\mathbb R$ such that $\int_{B_R^{\mathbb R}}|u|^q|x|^\theta\mathrm dx$ is finite, where $q\geq1$, $\theta\geq0$, and $B_R^{\mathbb R}\subset \mathbb R^N$ is the open Euclidean ball centered at the origin with radius $R\in(0,\infty]$. In \cite[Theorem 1.1]{zbMATH05978504}, de Figueiredo et al. established the following embedding results for Sobolev spaces of functions defined in Euclidean balls.
\begin{theorem}\label{theoDSM}
Let $W^{k,p}_{\mathbb R,\mathrm{rad}}(B_R^{\mathbb R})$ be the space of radially symmetric functions in $W^{k,p}_{\mathbb R}(B_R^{\mathbb R})$, where $k\geq1$ is an integer, $p\geq1$ is a real number, and $R\in(0,\infty)$.
\begin{flushleft}
$\mathrm{(1)}$ \justifying{Every function $u\in W^{k,p}_{\mathbb R,\mathrm{rad}}(B_R^{\mathbb R})$ is almost everywhere equal to a function $U$ in $C^{k-1}(\overline{B_R^{\mathbb R}}\backslash\{0\})$. Moreover, all partial derivatives of $U$ of order $k$ (in the classical sense) exist almost everywhere $|x|\in(0,R)$.}

\noindent$\mathrm{(2)}$ \justifying{If $N>kp$, then $W^{k,p}_{\mathbb R,\mathrm{rad}}(B_R^{\mathbb R})$ is continuously embedded in $L^q_\theta(B_R^{\mathbb R})$ for every $\theta\geq0$ and $1\leq q\leq \frac{(\theta+N)p}{N-kp}$.}

\noindent$\mathrm{(3)}$ \justifying{If $N=kp$ and $p>1$, then $W^{k,p}_{\mathbb R,\mathrm{rad}}(B_R^{\mathbb R})$ is compactly embedded in $L^q_\theta(B_R^{\mathbb R})$ for all $\theta\geq0$ and $1\leq q<\infty$.}

\noindent$\mathrm{(4)}$ If $N=kp$ and $p=1$, then $W^{N,1}_{\mathbb R,\mathrm{rad}}(B_R^{\mathbb R})$ is continuously embedded in $C(\overline{B_R^{\mathbb R}})$.
\end{flushleft}
\end{theorem}
To prove Theorem \ref{theoDSM}, the following radial lemmata were crucial to the argument. If $N>kp$, then there exists $C>0$ such that for all $u\in W^{k,p}_{\mathbb R,\mathrm{rad}}(B_R^{\mathbb R})$,
\begin{equation*}
|u(x)|\leq C\dfrac{\|u\|_{W^{k,p}(B_R^{\mathbb R})}}{|x|^{\frac{N-kp}p}},\quad\forall x\in \overline{B_R^{\mathbb R}}\backslash\{0\}.
\end{equation*}
If $N=kp$ and $p>1$, then there exists $C>0$ such that for all $u\in W^{k,p}_{\mathbb R,\mathrm{rad}}(B_R^{\mathbb R})$,
\begin{equation*}
|u(x)|\leq C\|u\|_{W^{k,p}(B_R^{\mathbb R})}\left[\left(\log\frac{R}{|x|}\right)^{\frac{p-1}p}+1\right],\quad\forall x\in \overline{B_R^{\mathbb R}}\backslash\{0\}.
\end{equation*}

In the context of the domain being the entire Euclidean space $\mathbb R^N$, recent work by the current authors, notably \cite[Theorem 1.1]{arXiv:2306.00194}, together with the item (1) of Theorem \ref{theoDSM}, yields the following result.
\begin{theorem}\label{theoB}
Let $W^{k,p}_{\mathbb R,\mathrm{rad}}(\mathbb R^N)$ be the space of radially symmetric functions in $W^{k,p}(\mathbb R^N)$, where $k\geq1$ is an integer and $p\geq1$ is a real number.
\begin{flushleft}
    $\mathrm{(1)}$ \justifying{Every function $u\in W^{k,p}_{\mathbb R,\mathrm{rad}}(\mathbb R^N)$ is almost everywhere equal to a function $U$ in $C^{k-1}(\mathbb R^N\backslash\{0\})$. Moreover, all partial derivatives of $U$ of order $k$ (in the classical sense) exist almost everywhere for $|x|\in(0,\infty)$.}
    
    \noindent$\mathrm{(2)}$ \justifying{If $N>kp$, then the following continuous embedding holds:
    \begin{equation*}
    W^{k,p}_{\mathbb R,\mathrm{rad}}(\mathbb R^N)\hookrightarrow L^q_{\theta}(\mathbb R^N),\quad p\leq q\leq p^*
    \end{equation*}
    is continuously embedded in $L^q_\theta(\mathbb R^N)$ for every $\theta\geq0$ and $p\leq q\leq \frac{(\theta+N)p}{N-kp}$.}
    
    \noindent$\mathrm{(3)}$ \justifying{If $N=kp$, then $W^{k,p}_{\mathbb R,\mathrm{rad}}(\mathbb R^N)$ is compactly embedded in $L^q_\theta(\mathbb R^N)$ for all $\theta\geq0$ and $p\leq q<\infty$.}
\end{flushleft}
\end{theorem}

To establish Theorem \ref{theoB}, we relied on \cite[Lemme II.1]{zbMATH03789294}, which provides the following decay property for functions in $W^{1,p}_{\mathbb R,\mathrm{rad}}(\mathbb R^N)$ ($N\geq2$),
\begin{equation}\label{dlrn}
|u(x)|\leq C\|u\|_{L^p(\mathbb R^N)}^{\frac{p-1}p}\|\nabla u\|^{\frac1p}_{L^p(\mathbb R^N)}|x|^{-\frac{N-1}p},\quad\forall x\in\mathbb R^N\backslash\{0\}.
\end{equation}

In Theorem \ref{theoB}, the estimate $\|u\|_{L^q_\theta(\mathbb R^N)} \leq C\|u\|_ {W^{k,p}(\mathbb{R}^N)}$ was obtained, utilizing the full norm of the Sobolev space $W^{k,p}(\mathbb{R}^N)$. For the estimate $\|u\|_{L^q_\theta(\mathbb R^N)} \leq C \|\nabla^k u\|_{L^p_\alpha(\mathbb R^N)}$, which involves only the norm in $L^p_{\alpha}(\mathbb{R}^N)$ of the $k$-th gradient, we refer to \cite{zbMATH06376695}.

The primary objective of this paper is to explore embedding results for $W^{k,p}_{\mathbb H,\mathrm{rad}}(B_R^{\mathbb H})$, the Sobolev space of radial functions defined on both bounded ball $B_R^{\mathbb H}$, (where $0<R<\infty$), and the entire hyperbolic space $B^{\mathbb H}_\infty=\mathbb{H}^N$. The essential contribution here is that we do not assume any boundary conditions when considering embeddings on bounded domains of hyperbolic spaces.
To achieve this, it was essential to prove radial lemmata and establish decay properties, both of which are intrinsically interesting.
We want to highlight that we are using the symbol $W^{k,p}_{\mathbb H,\mathrm{rad}}(B_R^{\mathbb H})$ and $W^{k,p}_{\mathbb R,\mathrm{rad}}(B_R^{\mathbb R})$ to represent the Sobolev spaces of radial functions on the hyperbolic space ball and on the Euclidean space ball, respectively.

We begin by introducing the fundamental concepts pertinent to our study.
Consider $\mathbb H^N$, the $N$-dimensional hyperbolic space represented by the Poincar\'e ball model:
\begin{equation*}
\mathbb H^N=\{x\in\mathbb R^N\colon|x|<1\}
\end{equation*}
equipped with the Riemannian metric
\begin{equation*}
g_{ij}=\left(\frac{2}{1-|x|^2}\right)^2\delta_{ij},
\end{equation*}
where $\delta_{ij}$ is the Kronecker delta. The distance between the origin and $x\in\mathbb H^N$ is determined by
\begin{equation*}
d(x)=\log\dfrac{1+|x|}{1-|x|}.
\end{equation*}
We denote $B_R^{\mathbb H}=\{x\in\mathbb H^N\colon d(x)<R\}$ as the hyperbolic geodesic ball with radius $R\in(0,\infty]$. Note that $B_R^{\mathbb H}=B_{\overline R}^{\mathbb R}$, where $\overline R=\tanh(\frac{R}{2})$ and, consequently, $R=\log\frac{1+\overline R}{1-\overline R}$. 

A function $u\colon B_R^{\mathbb H}\to\mathbb R$ is called \textit{radial in the hyperbolic space} if there exists a function $v\colon (0,R)\to\mathbb R$ such that $u(x)=v(d(x))$ for all $x\in B_R^{\mathbb H}$. Additionally, assuming $u\in L^1(B_R^{\mathbb H})$, the hyperbolic polar coordinates ($t=\log\frac{1+|x|}{1-|x|}$ and $|x|=\tanh\frac{t}2$) imply
\begin{equation}\label{eqpolar}
\int_{B_R^{\mathbb H}}u(x)\mathrm dV_g=\omega_{N-1}\int_0^Rv(t)\sinh^{N-1}(t)\mathrm dt,
\end{equation}
where $\omega_{N-1}$ is the $(N-1)$-dimensional surface measure of the unit sphere $\mathbb S^{N-1}\subset\mathbb R^N$.

For $\Omega\subset \mathbb H^N$, we define
\begin{equation*}
L^p_{\sinh^\theta}(\Omega)=\left\{u\colon\Omega\to\mathbb R\colon\int_\Omega|u|^p\sinh^\theta d(x)\mathrm dV_g<\infty\right\}
\end{equation*} 
and $W^{k,p}((0,R),\sinh^{N-1}(t))$ as the weighted Sobolev space consisting of all functions $v\colon(0,R)\to\mathbb R$ that have weak derivatives up to order $k$ and satisfy
\begin{equation*}
\|v\|_{W^{k,p}_{\sinh^{N-1}}}:=\left(\sum_{j=0}^k\int_0^R|v^{(j)}(t)|^p\sinh^{N-1}(t)\mathrm dt\right)^{\frac1p}<\infty.
\end{equation*}

\subsection{Description of the main results}

Assuming $u\in W^{k,p}_{\mathbb H,\mathrm{rad}}(B_R^{\mathbb H})$, our first main theorem establishes a relationship between $\nabla^ku$, the $k$-th order covariant derivative of $u$, and $v^{(k)}$, the $k$-th order derivative of $v$. This relationship is crucial for proving the embedding results for Sobolev spaces of functions defined in the hyperbolic space.

\begin{theo}\label{theouv}
If $u\in W^{k,p}_{\mathbb H,\mathrm{rad}}(B_R^{\mathbb H})$ with $R\in(0,\infty]$, then $v\in W^{k,p}((0,R),\sinh^{N-1}(t))$. Moreover,
\begin{equation}\label{eqineuv}
|\nabla^ku(x)|_g\geq|v^{(k)}(d(x))|,\quad\forall x\in B_R^{\mathbb H}\backslash\{0\}.
\end{equation}
\end{theo}
The main challenge in demonstrating the previous theorem lies in estimating the $k$-covariant derivative of $u$, which involves numerous terms resulting from classical derivative rules. For instance, just for $k=2$, we have:
\begin{align*}
\dfrac{\partial^2u}{\partial x_i\partial x_j}&=\left(\dfrac{2}{1-|x|^2}\right)^2v''(d(x))\frac{x_ix_j}{|x|^2}+\left(\dfrac{2}{1-|x|^2}\right)^2v'(d(x))\frac{x_ix_j}{|x|}\\
&\quad+\dfrac{2}{1-|x|^2}v'(d(x))\left(\dfrac{\delta_{ij}}{|x|}-\dfrac{x_ix_j}{|x|^3}\right),
\end{align*}
where $\delta_{ij}$ is the Kronecker delta. The number of terms increases with each additional derivative. Even in $\mathbb R^N$, this calculation is not straightforward, as demonstrated in the proof of \cite[Theorem 2.2]{zbMATH05978504}. Additionally, the covariant derivative includes a term depending on the Christoffel symbol of $\mathbb H^N$, as shown in \eqref{eqsf}. Therefore, it is evident that this proof is not elementary.

Motivated by Theorem \ref{theouv}, a natural question arises: if $u$ belongs to $W^{k,p}_{\mathbb H,\mathrm{rad}}(B_R^{\mathbb H})$, does $v$ belong to $W^{k,p}((0,R),\sinh^{N-1}(t))$ and vice-versa? Our next main theorem addresses this question.
\begin{theo}\label{theo10}
For each $R\in(0,\infty]$, $p\geq1$, and $k\geq1$ integer, we have
\begin{flushleft}
    $\mathrm{(1)}$ $W^{k,p}_{\mathbb H,\mathrm{rad}}(B_R^{\mathbb H})\hookrightarrow W^{k,p}((0,R),\sinh^{N-1}(t))$;
    
    \noindent$\mathrm{(2)}$ $W^{1,p}_{\mathbb H,\mathrm{rad}}(B_R^{\mathbb H})\equiv W^{1,p}((0,R),\sinh^{N-1}(t))$;
    
    \noindent$\mathrm{(3)}$ For $R\in(0,\infty)$ and $k\geq2$, we have $W^{k,p}_{\mathbb H,\mathrm{rad}}(B_R^{\mathbb H})\equiv W^{k,p}((0,R),\sinh^{N-1}(t))$ if, and only if, $N>(k-1)p$.
\end{flushleft}
\end{theo}

Analogous to Theorem \ref{theoDSM} for bounded balls in $\mathbb R^N$, we establish the embedding for the geodesic ball in hyperbolic space of $W^{k,p}_{\mathbb H,\mathrm{rad}}(B_R^\mathbb H)$ into $L^q_{\sinh^\theta}(B_R^\mathbb H)$ with our next main result.
\begin{theo}\label{theo11}
Assume $R\in(0,\infty)$, $p\geq1$ real, and $k\geq1$ integer.
\begin{flushleft}
    $\mathrm{(1)}$ \justifying{Every function $u\in W^{k,p}_{\mathbb H,\mathrm{rad}}(B_R^\mathbb H)$ is almost everywhere equal to a function $U$ in $C^{k-1}(\overline{B_R^\mathbb H}\backslash\{0\})$. Additionally, all partial derivatives of $U$ of order $k$ (in the classical sense) exist almost everywhere for $d(x)\in(0,R)$.}
    
    \noindent$\mathrm{(2)}$ \justifying{If $N>kp$, then $W^{k,p}_{\mathbb H,\mathrm{rad}}(B_R^{\mathbb H})$ is continuously embedded in $L^q_{\sinh^\theta}(B_R^{\mathbb H})$ for every $\theta\geq0$ and $1\leq q\leq \frac{(\theta+N)p}{N-kp}$.}
    
     \noindent$\mathrm{(3)}$ \justifying{If $N=kp$ and $p>1$, then $W^{k,p}_{\mathbb H,\mathrm{rad}}(B_R^{\mathbb H})$ is compactly embedded in $L^q_{\sinh^\theta}(B_R^{\mathbb H})$ for all $\theta\geq0$ and $1\leq q<\infty$.}
     
     \noindent$\mathrm{(4)}$ \justifying{If $N=kp$ and $p=1$, then $W^{N,1}_{\mathbb H,\mathrm{rad}}(B_R^{\mathbb H})$ is continuously embedded in $C(\overline{B_R^{\mathbb H}})$.}
\end{flushleft}
\end{theo}

\begin{remark}\label{remarkct}
    Assuming no weights in the Lebesgue space ($\theta=0$), the previous result implies that $W^{k,p}_{\mathbb H,\mathrm{rad}}(B_R^{\mathbb H}) \hookrightarrow L^q(B_R^{\mathbb H})$ for every $1 \leq q \leq \frac{Np}{N-kp}$. This follows from the classical Sobolev embedding theorem without the need for radial symmetry (see \cite[Theorem 10.1]{zbMATH01447265}).
\end{remark}

\begin{remark}
For an embbeding version from $W^{k,p}((0,R),\sinh^{N-1}(t))$ to $L^q_{\sinh^{\theta}}(0,R)$, we refer the reader to Theorem \ref{theoweighted}.
\end{remark}

We establish the following radial lemmata (see Proposition \ref{proprls} and Corollary \ref{corrlnkp}) to prove the Theorem \ref{theo11}. If $N>kp$, then there exists $C>0$ such that for all $u\in W^{k,p}_{\mathbb H,\mathrm{rad}}(B_R^\mathbb H)$,
\begin{equation}\label{rls}
|u(x)|\leq C\dfrac{\|u\|_{W_{\mathbb H}^{k,p}(B_R^{\mathbb H})}}{\sinh^\frac{N-kp}pd(x)},\quad\forall x\in \overline{B_R^{\mathbb H}}\backslash\{0\}.
\end{equation}
Equivalently, it also holds that
\begin{equation}\label{asfkj111}
|u(x)|\leq C\dfrac{\|u\|_{W^{k,p}_{\mathbb H}(B_R^{\mathbb H})}}{|x|^{\frac{N-kp}{p}}},\quad\forall x\in\overline{B_R^{\mathbb H}}\backslash\{0\}.
\end{equation}
If $N=kp$ and $p>1$, then there exists $C>0$ such that for all $u\in W^{k,p}_{\mathbb H,\mathrm{rad}}(B_R^{\mathbb H})$,
\begin{equation}
|u(x)|\leq C\|u\|_{W^{k,p}_{\mathbb H}(B_R^{\mathbb H})}\left[\left(\log\frac{\overline R}{|x|}\right)^{\frac{p-1}p}+1\right],\quad\forall x\in \overline{B_R^{\mathbb H}}\backslash\{0\}.\label{rlsc}
\end{equation}
These radial lemmata \eqref{asfkj111} and \eqref{rlsc} were expected because in the Euclidean case, it was proved the same growth type (see \cite[equations (1.1) and (1.2)]{zbMATH05978504}). The reason is that the vital part of these radial lemmata is the estimate near the origin, where the hyperbolic space is similar to the Euclidean space.

\begin{cor}\label{corcompact}
Assume $N>kp$ and $\theta\geq0$. If $1\leq q<\frac{p(N+\theta)}{N-kp}$, then $W^{k,p}_{\mathbb H,\mathrm{rad}}(B_R^{\mathbb H})\hookrightarrow L^q_{\sinh^{\theta}}(B_R^{\mathbb H})$ is compact.
\end{cor}

To prove embedding results of the Sobolev spaces into weighted Lebesgue spaces on hyperbolic spaces,   we need to obtain asymptotic decay when $d(x)\to\infty$. Inspired by \eqref{dlrn}, we prove a decay lemma for the hyperbolic space. More specifically, for any $u\in W^{1,p}_{\mathbb H,\mathrm{rad}}(\mathbb H^N)$ it holds
\begin{equation}\label{aksfn}
|u(x)|\leq \left(\frac{p}{\omega_{N-1}}\right)^{\frac1p}\|u\|_{L^p(\mathbb H^N)}^{\frac{p-1}p}\|\nabla u\|_{L^p(\mathbb H^N)}^{\frac1p}\sinh^{\frac{1-N}p}d(x),\quad\forall x\in\mathbb H^N\backslash\{0\}.
\end{equation}
Note that this growth on the right-hand side of \eqref{aksfn} is different from \eqref{dlrn}, the decay lemma on $\mathbb R^N$. That is because the hyperbolic distance expands faster as you move away from the origin than in the Euclidean case. Therefore, using this decay lemma, we can obtain our Sobolev embeddings when the domain is the whole $\mathbb H^N$. The statement of this result is as follows:

\begin{theo}\label{theo12}
Assume $\theta\geq0$, $p\geq1$ , and that $k\geq1$ is an integer.
\begin{flushleft}
$\mathrm{(1)}$ \justifying{Every function $u\in W^{k,p}_{\mathbb H,\mathrm{rad}}(\mathbb H^N)$ is almost everywhere equal to a function $U$ in $C^{k-1}(\mathbb H^N\backslash\{0\})$. In addition, all partial derivatives of $U$ of order $k$ (in the classical sense) exist almost everywhere for $d(x)\in(0,\infty)$.}

\noindent$\mathrm{(2)}$ If $N>kp$, then the following continuous embedding holds:
\begin{equation*}
W^{k,p}_{\mathbb H,\mathrm{rad}}(\mathbb H^N)\hookrightarrow L^q_{\sinh^\theta}(\mathbb H^N) \quad \text{if} \quad p\leq q\leq p^*,
\end{equation*}
where
\begin{equation*}
    p^*=\dfrac{(\theta+N)p}{N-kp}.
\end{equation*}
Moreover, it is a compact embedding if one of the following two conditions is fulfilled:
\begin{flushleft}
$\mathrm{(i)}$ $\theta=N-1$ and $p<q<p^*$;

\noindent$\mathrm{(ii)}$ $\theta<N-1$ and $p\leq q<p^*$.
\end{flushleft}
\noindent$\mathrm{(3)}$ \justifying{If $N=kp$, then the following continuous embedding holds:}
\begin{equation*}
W^{k,p}_{\mathbb H,\mathrm{rad}}(\mathbb H^N)\hookrightarrow L^q_{\sinh^\theta}(\mathbb H^N)\quad \text{if}\quad p\leq q<\infty.
\end{equation*}
Moreover, it is compact embedding if one of the following two conditions is fulfilled:
\begin{flushleft}
\noindent$\mathrm{(i)}$ $\theta=N-1$ and $q>p$;

\noindent$\mathrm{(ii)}$ $\theta<N-1$ and $q\geq p$.
\end{flushleft}
\end{flushleft}
\end{theo}
\begin{remark}
Similarly to Remark \ref{remarkct}, for $\theta=0$, the previous result implies that $W^{k,p}_{\mathbb H,\mathrm{rad}}(\mathbb H^N) \hookrightarrow L^q(\mathbb H^N)$ for every $p \leq q \leq \frac{Np}{N-kp}$. This follows from the classical Sobolev embedding theorem without the need for radial symmetry (see \cite[Theorem 3.2]{zbMATH01447265}).
\end{remark}

\begin{remark}
For an embedding version from $W^{k,p}((0,\infty),\sinh^{N-1}(t))$ to $L^q_{\sinh^\theta}(0,\infty)$, we refer to Theorem \ref{theo122}.
\end{remark}
 
While Sobolev-type inequalities play a crucial role in the study of partial differential equations,  some problems in geometric analysis involve the limiting cases of these inequalities. One such limiting case is the Trudinger-Moser inequality, which was first examined in \cite{zbMATH03232364, zbMATH03228089, zbMATH03261965, zbMATH03337983}. Motivated by issues related to the Gauss curvature equations, Moser refined Trudinger's inequality as follows:

\begin{equation}\label{moser}
\sup_{\{u\in W^{1,N}_0(\Omega)\colon \|\nabla u\|_{L^N(\Omega)}\leq1\}}\int_\Omega e^{\alpha|u|^{\frac{N}{N-1}}}\mathrm dx\left\{\begin{array}{ll}
     <\infty,&\mbox{if }\alpha\leq\alpha_N,  \\
     =\infty,&\mbox{if }\alpha>\alpha_N,
\end{array}\right.
\end{equation}
where $\alpha_N=N\omega_{N-1}^{\frac{1}{N-1}}$ and $\Omega\subset\mathbb R^N$ is a domain with finite volume.

The result mentioned above has been generalized in several directions. Adams \cite{zbMATH03337983}  obtained a generalized version of \eqref{moser} for Sobolev spaces with higher order derivatives. See also  \cite{zbMATH06252414} and references therein. This inequality \eqref{moser} has been extended to Riemannian manifolds (see, for example, \cite{zbMATH00490628,zbMATH05194595,zbMATH00447008,zbMATH07384223}). In particular, Trudinger-Moser and Adams-type inequalities have been established on hyperbolic spaces in recent years. In \cite{zbMATH05839583}, Mancini and Sandeep, for $N=2$, proved that a sharp Moser-Trudinger inequality holds on a conformal disc. Furthermore, Mancini, Sandeep, and Tintarev \cite{zbMATH06214509} proved a sharp Moser-Trudinger inequality in the hyperbolic space for all $N\geq2$. Lu and Tang \cite{zbMATH06256981} established both the critical and subcritical  Trudinger-Moser inequalities on bounded domains as well as on the entire hyperbolic spaces for $N\ge 2$, including the singular versions. Sharp Trudinger-Moser inequalities of exact growth have also been proved on hyperbolic space by Lu and Tang \cite{zbMATH05839583}. Wang and Ye \cite{zbMATH06029076}, along with Liang et al. \cite{zbMATH07247135}, have also established sharp Hardy-Trudinger-Moser inequalities on hyperbolic space.

For an Adams-type inequality in hyperbolic space, we refer to Karmakar and Sandeep \cite{zbMATH06536855},  Ng\^o and Nguyen \cite{zbMATH07318489} and Hardy-Adams inequalities by Li, Lu and Yang \cite{zbMATH06776234}, \cite{zbMATH06898920}, \cite{zbMATH07189073} on real hyperbolic spaces.
Lu and Yang studied these inequalities on complex hyperbolic space \cite{zbMATH07557162} and Flynn, Lu, and Yang on quaternionic and octonionic hyperbolic spaces \cite{zbMATH07835927}.
 
\medskip

We now return to consider the weighted  Trudinger-Moser and Adams type inequalities for symmetric functions on bounded balls in hyperbolic space when no boundary conditions is imposed. 
 A similar problem has been addressed by the current authors in \cite[Theorem 1.2]{arXiv:2302.02262}. Specifically, for $N=kp$ and $0<R<\infty$, Theorem \ref{theo11} states that $W^{k,p}_{\mathbb H,\mathrm{rad}}(B_R^{\mathbb H})\hookrightarrow L^q_{\sinh^\theta}(B_R^{\mathbb H})$ for all $1\leq q<\infty$ and $\theta\geq0$. However, it is known that $W^{1,p}_{\mathbb H,\mathrm{rad}}(B_R^{\mathbb H}) \not \hookrightarrow L^{\infty}(B_R^{\mathbb H})$, as demonstrated by the function $\phi(x)=\log(\log(\frac{e\overline R}{|x|}))$. This observation raises a natural question: what is the maximal possible growth for a function $g(s)$ such that $u\in W^{k,p}_{\mathbb H,\mathrm{rad}}(B_R^{\mathbb H})$ implies $\int_{B_R^{\mathbb H}} g(u) \, \mathrm dV_g < \infty$ when the functions $u$ are not assumed to have a boundary condition? 

Similar to classical works\cite{zbMATH03337983,zbMATH03261965,zbMATH04099653}, we demonstrate that the optimal growth is exponential. To formalize this, consider the following supremum 
\begin{equation}\label{e33}
\mathcal{L}_{\mu,\mathbb H}:= \sup_{\substack{u\in W^{k,p}_{\mathbb H,\mathrm{rad}}(B_R^{\mathbb H})\\ \|u\|_{W^{k,p}(B_R^{\mathbb H})}\leq 1}}\int_{B_R^{\mathbb H}}e^{\mu|u|^{\frac{p}{p-1}}}\sinh^\theta d(x) \mathrm dV_g.
\end{equation}

It is important to note that our analysis does not impose any boundary conditions for the functions and their higher order derivatives. The norm $\|u\|_{W^{k,p}_{\mathbb H}(B_R^{\mathbb H})}$ used here differs from that one utilized by Adams in \cite{zbMATH04099653}. Before we analyze \eqref{e33}, we should consider a similar problem for the weighted Sobolev space $W^{k,p}((0,R),\sinh^{N-1}(t))$, which consists in the study the following supremum
\begin{equation}\label{e34}
\mathcal{L}_\mu:=\sup_{\substack{v\in W^{k,p}((0,R),\sinh^{N-1}(t))\\ \|v\|_{W^{k,p}_{\sinh^{N-1}}}\leq1}}\int_0^Re^{\mu|v|^{\frac{p}{p-1}}}\sinh^\theta(t)\mathrm dt.
\end{equation}

As in celebrate works \cite{zbMATH03337983,zbMATH04099653},  we  determine the optimal limit constant $\mu_0>0$ such that \eqref{e34} is finite for $\mu<\mu_0$ and \eqref{e34} is infinite for $\mu>\mu_0$. The following main theorem establishes an Adams-type inequality for functions in the weighted Sobolev space $W^{k,p}((0,R),\sinh^{N-1}(t))$, which includes higher order derivatives and does not impose any boundary conditions.

\begin{theo}\label{theo16}
Let $\theta>-1$ and $W^{k,p}((0,R),\sinh^{N-1}(t))$ be the weighted Sobolev space with $N=kp$, $0<R<\infty$, and $p>1$.
\begin{flushleft}
    $\mathrm{(a)}$ For all $\mu\geq0$ and $u\in W^{k,p}((0,R),\sinh^{N-1}(t))$, we have $\exp(\mu|u|^{\frac{p}{p-1}})\in L^1_{\sinh^{\theta}}(0,R)$.\\
    $\mathrm{(b)}$ If $0\leq\mu<\mu_0$, then $\mathcal{L}_\mu$ is finite, where
    \begin{equation*}
        \mu_{0}:=(\theta+1)[(k-1)!]^{\frac{p}{p-1}}.
    \end{equation*}
    \justifying{Moreover, $\mathcal{L}_\mu$ is attained by a nonnegative function $u_0\in W^{k,p}((0,R),\sinh^{N-1}(t))$ with $\|u_0\|_{W^{k,p}_{\sinh^{N-1}}}=1$.}\\
    $\mathrm{(c)}$ If $\mu>\mu_0$, then $\mathcal{L}_\mu=\infty$.
\end{flushleft}
\end{theo}

By applying similar arguments presented in the previous theorem, we derive an Adams-type inequality for functions in $W^{k,p}(B_R^{\mathbb H})$ as our final main result. 
However, it is essential to note that the unboundedness of $\mathcal{L}_{\mu,\mathbb H}$ is established only for sufficiently large values of $\mu$. For a detailed explanation of this assumption, please refer to Remark \ref{remark}.

\begin{theo}\label{theo15}
Let $\theta>-N$ and $W^{k,p}_{\mathbb H,\mathrm{rad}}(B_R^{\mathbb H})$ be the radial Sobolev space with $N=kp$, $0<R<\infty$, and $p>1$.
\begin{flushleft}
$\mathrm{(a)}$ For all $\mu\geq0$ and $u\in W^{k,p}_{\mathbb H,\mathrm{rad}}(B_R^{\mathbb H}),$ we have $\exp(\mu|u|^{\frac{p}{p-1}})\in L^1_{\sinh^\theta}(B_R^{\mathbb H})$.\\
$\mathrm{(b)}$ If $0\leq\mu<\mu_{0,\mathbb H}$, then $\mathcal{L}_{\mu,\mathbb H}$ is finite, where
\begin{equation*}
\mu_{0,\mathbb H}:=(\theta+N)[\omega_{N-1}^{\frac{1}{p}}(k-1)!]^{\frac{p}{p-1}}.
\end{equation*}
\justifying{Moreover, assuming $\theta\geq0$, $\mathcal{L}_{\mu,\mathbb H}$ is attained by a nonnegative function $u_0\in W^{k,p}_{\mathbb H,\mathrm{rad}}(B_R^{\mathbb H})$ with $\|u_0\|_{W^{k,p}_{\mathbb H}(B_R^{\mathbb H})}=1$.}\\
$\mathrm{(c)}$ If $\mu>C_0^{\frac{p}{p-1}}(\theta+N)[(k-1)!]^{\frac{p}{p-1}}$, where $C_0$ is the constant associated with the continuous embedding $W^{k,p}((0,R),\sinh^{N-1}(t))\hookrightarrow W^{k,p}_{\mathbb H,\mathrm{rad}}(B_R^{\mathbb H})$ as provided in Theorem \ref{theo10}, then $\mathcal{L}_{\mu,\mathbb H}=\infty$.
\end{flushleft}
\end{theo}

\subsection{General strategy of the proofs}

One of the main challenges in our paper is the proof of Theorem \ref{theouv}, which requires obtaining an expression for $v^{(j)}$ in terms of the coordinates of $\nabla^ju$. We first establish this expression for smooth functions and then extend it to weak derivatives using the divergence theorem. The key difference between our approach and the Euclidean space lies in the presence of the Christoffel symbols in the recursive formula for the covariant derivative, which significantly complicates the proof.

For Theorem \ref{theo10}, the items (i) and (ii) follow straightforwardly from Theorem \ref{theouv}. However, the proof of item (iii) is not trivial, and we introduce a novel approach. Specifically, we require an expression of $\nabla^ju$ in terms of $v^{(j)}$. While we only obtain a recursive formula, this suffices for our results. Additionally, we restrict our domain to a conical annular sector to show that $d(x)\notin W^{k,p}((0,R),\sinh^{N-1}(t))$ whenever $N\leq(k-1)p$.

With Theorems \ref{theouv} and \ref{theo10} in place, we proceed to prove several radial lemmata and a decay lemma. These inequalities allow us to establish Theorems \ref{theo11}, \ref{theo12}, \ref{theo16}, and \ref{theo15}. Furthermore, to show that the supremum is infinite in item (c) of Theorem \ref{theo16} and item (c) of \ref{theo15}, we adapt the classical sequence introduced by Adams \cite{zbMATH04099653}.

\subsection{Organization of the paper}
The rest of the paper is divided as follows. In Section \ref{sec2}, we introduce essential definitions and properties of the hyperbolic space, covariant derivatives, and Sobolev spaces on Riemannian Manifolds. Section \ref{sec3} contains the demonstration that, for radial functions, the covariant derivative is greater than or equal to the derivative of the radial representative in the proof of Theorem \ref{theouv}. Section \ref{sec4} offers a comprehensive study of the relationships between the spaces $W^{k,p}_{\mathbb H,\mathrm{rad}}(B_R^{\mathbb H})$ and $W^{k,p}((0,R),\sinh^{N-1}(t))$ as seen in the proof of Theorem \ref{theo10}. In Section \ref{sec5}, we develop radial lemmata and their implications, including the proof of Theorem \ref{theo11} and Corollary \ref{corcompact}. Section \ref{sec6} is dedicated to proving a decay lemma (Lemma \ref{decaylemma}) and, using it, we achieve the Sobolev embedding on the entire $\mathbb H^N$, as detailed in the proof of Theorem \ref{theo12}. In Section \ref{sec7}, we study the Adams-type inequality in the proofs of Theorems \ref{theo16} and \ref{theo15}. Finally, Section \ref{sec8} provides some concluding remarks on the challenges associated with the Adams-type inequality and discusses potential applications of the results presented in this paper.

\section{Preliminaries}\label{sec2}

In this section, we introduce some notations and established results related to hyperbolic spaces, specifically within the Poincar\'e ball model. The Poincaré ball model for $\mathbb{H}^N$ is defined as $\{x\in \mathbb{R}^N\colon|x|< 1\}$, equipped with the Riemannian metric
$g_{ij}=(\frac{2}{1-|x|^2})^2\delta_{ij}$. The inverse of this Riemannian metric and the corresponding Christoffel symbols are given by $g^{ij}=(\frac{1-|x|^2}{2})^2 \delta_{ij}$ and
\begin{equation}\label{symChris}
\Gamma^k_{ij}(x)=\dfrac{2}{1-|x|^2}\left(x_i\delta_{jk}+x_j\delta_{ik}-x_k\delta_{ij}\right),\quad\forall i,j,k=1,\ldots,N.    
\end{equation}
The distance between the origin and $x\in\mathbb H^N$ is determined by
\begin{equation*}
d(x)=\log\dfrac{1+|x|}{1-|x|}.
\end{equation*}
The volume element is given by
\begin{equation*}
\mathrm dV_g=\left(\dfrac{2}{1-|x|^2}\right)^N\mathrm dx.
\end{equation*}

Next, we introduce definitions from \cite{zbMATH01447265} to define Sobolev spaces on Riemannian manifolds. The space $T_k(T_xM)$ consists of all $k$ linear $k$-covariant tensors
\begin{equation*}
\eta\colon \underbrace{T_xM\times\cdots\times T_xM}_\text{$k$-times}\rightarrow\mathbb R.
\end{equation*}
For a fixed chart of $M$ at $x$, the set
\begin{equation*}
    \left\{dx_{i_1}\otimes\cdots\otimes dx_{i_k}\right\}_{i_1\ldots i_k}
\end{equation*}
is a base for $T_k(T_xM)$, where $\otimes$ is the tensor product. The components of a $k$-covariant tensor $\eta$ are denoted by $\eta_{i_1\ldots i_j}$. Thus,
\begin{equation*}
\eta=\sum_{i_1\ldots i_j=1}^N\eta_{i_1\ldots i_j}dx_{i_1}\otimes\cdots\otimes dx_{i_j}.
\end{equation*}
A map $\eta\colon M\to \amalg_{x\in M}T_k(T_xM)$ is called a $k$-covariant tensor field if $\eta(x)\in T_k(T_xM)$ for all $x\in M$. We say that $\eta$ is of class $C^{j}$ if it is of class $C^j$ from the manifold $M$ to the manifold $\amalg_{x\in M}T_k(T_xM)$ (equivalently, $\eta_{i_1\ldots i_k}\colon M\to\mathbb R$ is of class $C^j$ for all $i_1,\ldots,i_k$). The metric $g$ induces a metric in the space of $k$-covariant tensor field as follows:
\begin{equation*}
\langle \eta,\nu\rangle_g:=\sum_{i_1\ldots i_kj_1\ldots j_k}g^{i_1j_1}\cdots g^{i_kj_k}\eta_{i_1\ldots i_k}\nu_{j_1\ldots j_k},
\end{equation*}
where $\eta$ and $\nu$ are $k$-covariant tensor fields.

Given a $k$-covariant tensor field $\eta$ of class $C^{j+1}$, the covariant derivative $\nabla\eta$ is defined as a $(k+1)$-covariant tensor field of class $C^j$ with components given by
\begin{equation*}
\left(\nabla \eta\right)(x)_{i_1\ldots i_k}=\dfrac{\partial \eta_{i_2\ldots i_k}}{\partial x_{i_1}}(x)-\sum_{\ell=2}^k\sum_{\alpha=1}^N\Gamma^\alpha_{i_1i_\ell}\eta(x)_{i_2\ldots i_{\ell-1}\alpha i_{\ell+1}\ldots i_k}.
\end{equation*}
In the context of the Poincaré disk model, using \eqref{symChris}, we have
\begin{equation}\label{eqsf}
\left(\nabla\eta\right)(x)_{i_1\ldots i_k}=\dfrac{\partial \eta_{i_2\ldots i_k}}{\partial x_{i_1}}(x)-\dfrac{2x_{i_1}}{1-|x|^2}(k-1)\eta_{i_2\ldots i_k}(x)-\dfrac{2}{1-|x|^2}\sum_{\ell=2}^kx_{i_\ell}\eta_{i_2\ldots i_{\ell-1}i_1i_{\ell+1}\ldots i_k}.
\end{equation}

For a smooth function $u\colon M\to\mathbb R$, we denote by $\nabla^ku$ the $k$-covariant derivative of $u$, and by $|\nabla^ku|_g$ the norm of $\nabla^ku$, defined as
\begin{equation*}
|\nabla^ku|_g^2=\sum_{i_1\ldots i_kj_1\ldots j_k=1}^Ng^{i_1j_1}\cdots g^{i_kj_k}\left(\nabla^ku\right)_{i_1\ldots i_k}\left(\nabla^ku\right)_{j_1\ldots j_k}.
\end{equation*}
For the metric of the Poincaré disk model, we have
\begin{equation}\label{eqnormkcov}
|\nabla^ku|_g^2=\left(\dfrac{1-|x|^2}2\right)^{2k}\sum_{i_1\ldots i_k}\left(\nabla^ku\right)^2_{i_1\ldots i_k}.
\end{equation}

Given a positive integer $k$ and real number $p\geq1$, we define
\begin{equation*}
\mathfrak C^{k,p}(M)=\left\{u\in C^\infty(M)\colon\int_M|\nabla^ju|_g^p\mathrm dV_g<\infty,\quad\forall j=0,\ldots,k\right\}
\end{equation*}
and $W^{k,p}(M)$ the Sobolev space on Riemannian manifold as the completion of $\mathfrak C^{k,p}(M)$ under the norm $\|\cdot\|_{W^{k,p}(M)}$, where
\begin{equation*}
    \|u\|_{W^{k,p}(M)}=\left(\sum_{j=0}^k\int_M|\nabla^ju|_g^p\mathrm dV_g\right)^{\frac1p},\quad\mbox{for }u\in \mathfrak C^{k,p}(M).
\end{equation*}

The space $W^{k,p}(M)$ can be interpreted as the space of all functions $u\in L^p(M)$ such that there exists a Cauchy sequence $(u_n)$ in $\mathfrak C^{k,p}(M)$ with $u_n\to u$ in $L^p(M)$. Therefore, we define $\nabla^ju$ as the $L^p$-limit of the sequence $(\nabla^ju_n)$, and the norm $\|\cdot\|_{W^{k,p}(M)}$ is defined as given above.

\section{Proof of \texorpdfstring{$|\nabla^ku(x)|_g\geq|v^{(k)}(d(x))|$}{}}\label{sec3}

This section assumes $R\in(0,\infty]$. The main objective of this section is to prove Theorem \ref{theouv}. Here, we outline the proof's strategy. Remember that $u(x)=v(d(x))$.

First, we suppose $u$ is a smooth function to derive an expression for all classical derivatives of $v$ in terms of the classical derivatives of $u$. This result will be established in Lemma \ref{lemmaagh}. The proof involves induction and some manipulations of the recursive formula for the $k$-th covariant derivative.

Next, we assume $u$ has weak derivatives and obtain an expression for the weak derivatives of $v$, as stated in Lemma \ref{lemmaagh1}. This step utilizes the result from the previous step applied to test functions, alongside the divergence theorem and a technical remark (Remark \ref{remarktech}) derived in the last step.

Finally, using the expression for the weak derivatives of $v$ obtained in Lemma \ref{lemmaagh1} and applying some classical inequalities, we conclude the proof of Theorem \ref{theouv}, demonstrating that $|\nabla^ku(x)|_g\geq|v^{(k)}(d(x))|$.

\begin{lemma}\label{lemmaagh}
Let $u\in C^\infty(B_R^{\mathbb H})$ be a radial function. Then $v\in C^\infty(0,R)$ with, for each $j\in\mathbb N$,
\begin{equation}\label{equvcov}
v^{(j)}(d(x))=\left(\dfrac{1-|x|^2}2\right)^j\sum_{i_1\ldots i_j=1}^N\left(\nabla^ju\right)_{i_1\cdots i_j}(x)\dfrac{x_{i_1}\cdots x_{i_j}}{|x|^j},\quad\forall x\in B_R^{\mathbb H}\backslash\{0\}.
\end{equation}
\end{lemma}
\begin{proof}
Let us prove \eqref{equvcov} by induction on $j$. By $u(x)=v(d(x))$ we have
\begin{equation*}
\dfrac{\partial u}{\partial x_i}=v'(d(x))\dfrac{2}{1-|x|^2}\dfrac{x_i}{|x|}.
\end{equation*}
By multiplying each term by $x_i$ and summing for $i$ from $1$ to $N$, we can conclude that
\begin{equation*}
v'(d(x))=\dfrac{1-|x|^2}{2}\sum_{i=1}^N\left(\nabla u\right)_i\dfrac{x_i}{|x|}.
\end{equation*}
This concludes \eqref{equvcov} for $j=1$.

By induction hypothesis, we have
\begin{equation*}
v^{(j)}(d(x))=\left(\dfrac{1-|x|^2}2\right)^j\sum_{i_2\ldots i_{j+1}=1}^N\left(\nabla^ju\right)_{i_2\ldots i_{j+1}}\dfrac{x_{i_2}\cdots x_{i_{j+1}}}{|x|^j}.
\end{equation*}
Applying $\frac{\partial}{\partial x_{i_1}}$ on both sides,
\begin{align}
v^{(j+1)}(d(x))&\dfrac{2}{1-|x|^2}\dfrac{x_{i_1}}{|x|}=j\left(\dfrac{1-|x|^2}2\right)^{j-1}(-
x_{i_1})\sum_{i_2\ldots i_{j+1}=1}^N\left(\nabla^ju\right)_{i_2\ldots i_{j+1}}\dfrac{x_{i_2}\cdots x_{i_{j+1}}}{|x|^j}\nonumber\\
&\quad+\left(\dfrac{1-|x|^2}{2}\right)^j\sum_{i_2\ldots i_{j+1}=1}^N\dfrac{\partial \left(\nabla^ju\right)_{i_2\ldots i_{j+1}}}{\partial x_{i_1}}\dfrac{x_{i_2}\cdots x_{i_{j+1}}}{|x|^j}\nonumber\\
&\quad+\left(\dfrac{1-|x|^2}{2}\right)^j\sum_{i_2\ldots i_{j+1}=1}^N\left(\nabla^ju\right)_{i_2\ldots i_{j+1}}\left(\prod_{\ell=2}^{j+1}\frac{x_{i_\ell}}{|x|}\right)\sum_{\ell=2}^{j+1}\left(\frac{\delta_{i_\ell i_1}}{x_{i_\ell}}-\frac{x_{i_1}}{|x|^2}\right)\label{eqvgh1}.
\end{align}
Let us denote the three terms in the right hand of \eqref{eqvgh1} by, respectively, $I_1$, $I_2$, and $I_3$. Note that
\begin{equation}\label{eqvgh2}
\sum_{i_1=1}^NI_1x_{i_1}=-j|x|^2\left(\frac{1-|x|^2}2\right)^{j-1}\sum_{i_2\ldots i_{j+1}=1}^N\left(\nabla^ju\right)_{i_2\ldots i_{j+1}}\dfrac{x_{i_2}\cdots x_{i_{j+1}}}{|x|^j}
\end{equation}
and
\begin{equation}\label{eqvgh3}
\sum_{i_1=1}^NI_2x_{i_1}=|x|\left(\dfrac{1-|x|^2}{2}\right)^j\sum_{i_1\ldots i_{j+1}=1}^N\dfrac{\partial \left(\nabla^ju\right)_{i_2\ldots i_{j+1}}}{\partial x_{i_1}}\dfrac{x_{i_1}\cdots x_{i_{j+1}}}{|x|^{j+1}}.
\end{equation}
On the other hand, by
\begin{equation*}
\sum_{i_1=1}^Nx_{i_1}\sum_{\ell=2}^{j+1}\left(\frac{\delta_{i_\ell i_1}}{x_{i_\ell}}-\frac{x_{i_1}}{|x|^2}\right)=0,
\end{equation*}
we have $\sum_{i_1=1}^NI_3x_{i_1}=0$. Using this together with \eqref{eqvgh2} and \eqref{eqvgh3} on \eqref{eqvgh1}, we obtain
\begin{align}
v^{(j+1)}(d(x))&=\left(\dfrac{1-|x|^2}2\right)^{j+1}\sum_{i_1\ldots i_{j+1}=1}^N\dfrac{\partial \left(\nabla^ju\right)_{i_2\ldots i_{j+1}}}{\partial x_{i_1}}\dfrac{x_{i_1}\cdots x_{i_{j+1}}}{|x|^{j+1}}\nonumber\\
&\quad-j|x|\left(\frac{1-|x|^2}2\right)^{j}\sum_{i_2\ldots i_{j+1}=1}^N\left(\nabla^ju\right)_{i_2\ldots i_{j+1}}\dfrac{x_{i_2}\cdots x_{i_{j+1}}}{|x|^j}.\label{safojabsdfkja1}
\end{align}
From the definition of $j+1$ covariant derivative (see \eqref{eqsf}),
\begin{align}
v^{(j+1)}(d(x))&=\left(\dfrac{1-|x|^2}2\right)^{j+1}\sum_{i_1\ldots i_{j+1}=1}^N\left(\nabla^{j+1}u\right)_{i_1\ldots i_{j+1}}\dfrac{x_{i_1}\cdots x_{i_{j+1}}}{|x|^{j+1}}\nonumber\\
&\quad+\left(\dfrac{1-|x|^2}2\right)^{j+1}\sum_{\alpha i_1\ldots i_{j+1}=1}^N\sum_{\ell=2}^{j+1}\left(\nabla^ju\right)_{i_2\ldots i_{\ell-1}\alpha i_{\ell+1}\ldots i_{j+1}}\Gamma_{i_1i_\ell}^\alpha\dfrac{x_{i_1}\cdots x_{i_{j+1}}}{|x|^{j+1}}\nonumber\\
&\quad-j|x|\left(\frac{1-|x|^2}2\right)^{j}\sum_{i_2\ldots i_{j+1}=1}^N\left(\nabla^ju\right)_{i_2\ldots i_{j+1}}\dfrac{x_{i_2}\cdots x_{i_{j+1}}}{|x|^j}.\label{safojabsdfkja2}
\end{align}
The proof of \eqref{eqineuv} is completed by showing that
\begin{align}
I_4&:=\dfrac{1-|x|^2}2\sum_{\alpha i_1\ldots i_{j+1}=1}^N\sum_{\ell=2}^{j+1}\left(\nabla^ju\right)_{i_2\ldots i_{\ell-1}\alpha i_{\ell+1}\ldots i_{j+1}}\Gamma_{i_1i_\ell}^\alpha\dfrac{x_{i_1}\cdots x_{i_{j+1}}}{|x|^{j+1}}\nonumber\\
&=j|x|\sum_{i_2\ldots i_{j+1}=1}^N\left(\nabla^ju\right)_{i_2\ldots i_{j+1}}\dfrac{x_{i_2}\cdots x_{i_{j+1}}}{|x|^j}.\label{eqvgh5}
\end{align}
Note that, from \eqref{symChris},
\begin{align}
I_4&=\sum_{i_1\ldots i_{j+1}=1}^N\sum_{\ell=2}^{j+1}x_{i_1}\left(\nabla^ju\right)_{i_2\cdots i_{j+1}}\dfrac{x_{i_1}\cdots x_{i_{j+1}}}{|x|^{j+1}}\nonumber\\
&\quad+\sum_{i_1\ldots i_{j+1}=1}^N\sum_{\ell=2}^{j+1}x_{i_\ell}\left(\nabla^ju\right)_{i_2\ldots i_{\ell-1}i_1 i_{\ell+1}\ldots i_{j+1}}\dfrac{x_{i_1}\cdots x_{i_{j+1}}}{|x|^{j+1}}\nonumber\\
&\quad-\sum_{\alpha i_2\ldots i_{j+1}=1}^N\sum_{\ell=2}^{j+1}x_\alpha x_{i_\ell}\left(\nabla^ju\right)_{i_2\ldots i_{\ell-1}\alpha i_{\ell+1}\ldots i_{j+1}}\dfrac{x_{i_2}\cdots x_{i_{j+1}}}{|x|^{j+1}}.\label{eqvgh6}
\end{align}
For a fixed $\ell=2,\ldots,j+1$ we can switch $i_1$ by $i_\ell$ to obtain
\begin{align}
\sum_{i_1\ldots i_{j+1}=1}^Nx_{i_\ell}\left(\nabla^ju\right)_{i_2\ldots i_{\ell-1}i_1 i_{\ell+1}\ldots i_{j+1}}\dfrac{x_{i_1}\ldots x_{i_{j+1}}}{|x|^{j+1}}&=\sum_{i_1\ldots i_{j+1}=1}^Nx_{i_1}\left(\nabla^ju\right)_{i_2\ldots i_{j+1}}\dfrac{x_{i_1}\cdots x_{i_{j+1}}}{|x|^{j+1}}\nonumber\\
&=|x|\sum_{i_2\ldots i_{j+1}=1}^N\left(\nabla^ju\right)_{i_2\ldots i_{j+1}}\dfrac{x_{i_2}\cdots x_{i_{j+1}}}{|x|^j}.\label{eqvgh7}
\end{align}
Analogously, fixed $\ell=2,\ldots,j+1$ and switching $i_\ell$ by $\alpha$ we have
\begin{align}
\sum_{\alpha i_2\ldots i_{j+1}=1}^Nx_\alpha x_{i_\ell}\left(\nabla^ju\right)_{i_2\ldots i_{\ell-1}\alpha i_{\ell+1}\ldots i_{j+1}}\dfrac{x_{i_2}\cdots x_{i_{j+1}}}{|x|^{j+1}}&=\sum_{\alpha i_2\ldots i_{j+1}=1}^Nx_\alpha^2 \left(\nabla^ju\right)_{i_2\ldots i_{j+1}}\dfrac{x_{i_2}\cdots x_{i_{j+1}}}{|x|^{j+1}}\nonumber\\
&=|x|\sum_{i_2\ldots i_{j+1}=1}^N\left(\nabla^ju\right)_{i_2\ldots i_{j+1}}\dfrac{x_{i_2}\cdots x_{i_{j+1}}}{|x|^j}.\label{eqvgh8}
\end{align}
Therefore, \eqref{eqvgh7} and \eqref{eqvgh8} in \eqref{eqvgh6} imply \eqref{eqvgh5} which concludes \eqref{equvcov}.
\end{proof}

\begin{remark}\label{remarktech}
By the previous proof, using \eqref{safojabsdfkja1}, \eqref{safojabsdfkja2}, and \eqref{eqvgh5} we have that 
\begin{align}
\sum_{i_1\ldots i_{j+1}=1}^N\!\left(\nabla^{j+1}u\right)_{i_{1}\ldots i_{j+1}}\!x_{i_1}\cdots x_{i_{j+1}}\!&=\!\sum_{i_1\ldots i_{j+1}=1}^N\dfrac{\partial\left(\nabla^j u\right)_{i_2\ldots i_{j+1}}}{\partial x_{i_{j+1}}}x_{i_1}\cdots x_{i_{j+1}}\nonumber\\
&\quad -j|x|^2\frac{2}{1-|x|^2}\!\sum_{i_2\ldots i_{j+1}=1}^N\!\left(\nabla^j u\right)_{i_1\ldots i_j}\!x_{i_2}\cdots x_{i_{j+1}}.\label{eqvgh51}
\end{align}
\end{remark}

\begin{lemma}\label{lemmaagh1}
Let $u\colon B_R^{\mathbb H}\to\mathbb R$ be a radial function with weak derivatives up to order $j$. Then $v$ has weak derivatives up to order $j$ and it is given by
\begin{equation}\label{eq310}
v_j(t):=\left(\frac{1-|x|^2}2\right)^j\sum_{i_1\ldots i_j=1}^N(\nabla^ju)_{i_1\cdots i_k}(x)\dfrac{x_{i_1}\cdots x_{i_j}}{|x|^j}
\end{equation}
where $t=d(x)$.
\end{lemma}
\begin{proof}
Fix $\varphi\in C^\infty_0(0,R)$. Then, using Lemma \ref{lemmaagh} on $\varphi$ and $\sinh(d(x))=\frac{2|x|}{(1-|x|^2)}$,
\begin{align*}
\int_0^R v(t)\varphi^{(j)}(t)dt&=\frac{1}{\omega_{N-1}}\int_{B_R^{\mathbb H}}u(x)\varphi^{(j)}(d(x))\frac{1}{\sinh^{N-1}(d(x))}\mathrm dV_g\\
&=\frac{1}{\omega_{N-1}}\int_{B_R^{\mathbb H}}\left(\dfrac{1-|x|^2}2\right)^{N-1}\frac{1}{|x|^{N-1}}u(x)\varphi^{(j)}(d(x))\mathrm dV_g.
\end{align*}
Denote $\psi(x)=\varphi(d(x))$. Using \eqref{equvcov} for $\varphi$, we have
\begin{equation*}
\varphi^{(j)}(d(x))=\left(\dfrac{1-|x|^2}2\right)^j\sum_{i_1\ldots i_j=1}^N\left(\nabla^j\psi\right)_{i_1\ldots i_j}\dfrac{x_{i_1}\cdots x_{i_j}}{|x|^j}.
\end{equation*}
If
\begin{align}
\int_{B_R^{\mathbb H}}&\dfrac{u(x)}{|x|^{N+j-1}}\left(\dfrac{1-|x|^2}2\right)^{N+j-1}\sum_{i_1\ldots i_j=1}^N\left(\nabla^j\psi\right)_{i_1\ldots i_j}x_{i_1}\cdots x_{i_j}\mathrm dV_g\nonumber\\
&=(-1)^j\int_{B_R^{\mathbb H}}\dfrac{\psi(x)}{|x|^{N+j-1}}\left(\dfrac{1-|x|^2}2\right)^{N+j-1}\sum_{i_1\ldots i_j=1}^N\left(\nabla^ju\right)_{i_1\ldots i_j}x_{i_1}\cdots x_{i_j}\mathrm dV_g\label{eqintpart}
\end{align}
then
\begin{align*}
\int_0^Rv(t)\varphi^{(j)}(t)\mathrm dt&=\frac{1}{\omega_{N-1}}\int_{B_R^{\mathbb H}}\dfrac{u(x)}{|x|^{N+j-1}}\left(\dfrac{1-|x|^2}2\right)^{N+j-1}\sum_{i_1\ldots i_j=1}^N\left(\nabla^j\psi\right)_{i_1\ldots i_j}x_{i_1}\cdots x_{i_j}\mathrm dV_g\\
&=\frac{(-1)^j}{\omega_{N-1}}\int_{B_R^{\mathbb H}}\dfrac{\psi(x)}{|x|^{N+j-1}}\left(\dfrac{1-|x|^2}2\right)^{N+j-1}\sum_{i_1\ldots i_j=1}^N\left(\nabla^ju\right)_{i_1\ldots i_j}x_{i_1}\cdots x_{i_j}\mathrm dV_g\\
&=(-1)^j\int_0^Rv^{(j)}(t)\varphi(t)\mathrm dt.
\end{align*}

To obtain \eqref{eq310}, we are left with the task of proving \eqref{eqintpart} which is a direct consequence of
\begin{align*}
&(-1)^\ell\sum_{i_1\ldots i_j=1}^N\int_{B_R^{\mathbb H}}\dfrac{x_{i_1}\cdots x_{i_j}}{|x|^{N+j-1}}\left(\dfrac{1-|x|^2}2\right)^{N+j-1}\left(\nabla^{j-\ell}\psi\right)_{i_{\ell+1}\ldots i_{j}}(\nabla^\ell u)_{i_\ell\ldots i_1}\mathrm dV_g\\
&\quad=\sum_{i_1\ldots i_j=1}^N\int_{B_R^{\mathbb H}}\dfrac{x_{i_1}\cdots x_{i_j}}{|x|^{N+j-1}}\left(\dfrac{1-|x|^2}2\right)^{N+j-1}\left(\nabla^{j}\psi\right)_{i_1\ldots i_{j}}u(x)\mathrm dV_g,\quad\forall \ell=0,\ldots,j.
\end{align*}
Fixed $\ell=0,\ldots,j-1$, it is enough to show that
\begin{align}
I_5&:=\sum_{i_1\ldots i_j=1}^N\int_{B_R^{\mathbb H}}\dfrac{x_{i_1}\cdots x_{i_j}}{|x|^{N+j-1}}\left(\dfrac{1-|x|^2}2\right)^{N+j-1}\left(\nabla^{j-\ell}\psi\right)_{i_{\ell+1}\ldots i_{j}}(\nabla^\ell u)_{i_\ell\ldots i_1}\mathrm dV_g\nonumber\\
&=-\sum_{i_1\ldots i_j=1}^N\int_{B_R^{\mathbb H}}\dfrac{x_{i_1}\cdots x_{i_j}}{|x|^{N+j-1}}\left(\dfrac{1-|x|^2}2\right)^{N+j-1}\left(\nabla^{j-\ell-1}\psi\right)_{i_{\ell+2}\ldots i_{j}}(\nabla^{\ell+1}u)_{i_{\ell+1}\ldots i_1}\mathrm dV_g.\label{intpart}
\end{align}
Note that \eqref{eqvgh51} implies that
\begin{align*}
\sum_{i_{\ell+1}\ldots i_{j}=1}^N\left(\nabla^{j-\ell}\psi\right)_{i_{\ell+1}\ldots i_{j}}&x_{i_{\ell+1}}\cdots x_{i_{j}}=\sum_{i_{\ell+1}\ldots i_{j}=1}^N\dfrac{\partial(\nabla^{j-\ell-1}\psi)_{i_{\ell+2}\ldots i_{j}}}{\partial x_{i_{\ell+1}}}x_{i_{\ell+1}}\cdots x_{i_{j}}\\
&-(j-\ell-1)|x|^2\dfrac{2}{1-|x|^2}\sum_{i_{\ell+2}\ldots i_{j}=1}^N\left(\nabla^{j-\ell-1} \psi\right)_{i_{\ell+2}\ldots i_{j}}x_{i_{\ell+2}}\cdots x_{i_{j}}.
\end{align*}
Then
\begin{align}
&I_5=\sum_{i_1\ldots i_j=1}^N\int_{B_R^{\mathbb H}}\dfrac{x_{i_1}\cdots x_{i_j}}{|x|^{N+j-1}}\left(\dfrac{1-|x|^2}2\right)^{N+j-1}\dfrac{\partial\left(\nabla^{j-\ell-1}\psi\right)_{i_{\ell+2}\ldots i_{j}}}{\partial x_{\ell+1}}(\nabla^\ell u)_{i_\ell\ldots i_1}\mathrm dV_g\label{gkal}\\
&-\!(j-\ell-1)\!\sum_{i_1\ldots i_j=1}^N\int_{B_R^{\mathbb H}}\dfrac{x_{i_1}\cdots x_{i_j}}{|x|^{N+j-1}}\left(\dfrac{1-|x|^2}2\right)^{N+j-2}\!x_{i_{\ell+1}}\left(\nabla^{j-\ell-1}\psi\right)_{i_{\ell+2}\ldots i_{j}}(\nabla^\ell u)_{i_\ell\ldots i_1}\mathrm dV_g.\nonumber
\end{align}
Since $\psi$ has compact support, by the divergence theorem,
\begin{equation}\label{eqdivtheo2}
\int_{B_{\overline R}^{\mathbb R}}w\mathrm{div}(V)\mathrm dV_g=-\int_{B_{\overline R}^{\mathbb R}}\langle\nabla w,V\rangle_g\mathrm dV_g,\quad\forall w\in C^\infty(B_{\overline R}^{\mathbb R}),
\end{equation}
where $B_{\overline R}^{\mathbb R}=B_R^{\mathbb H}$ and $V=(\nabla^{j-\ell-1}\psi)_{i_{\ell+2}\ldots i_j}(\frac{1-|x|^2}2)^{N+j-\ell-1}\partial_{i_{\ell+1}}$ with $i_{\ell+1}$ fixed. Note that
\begin{align*}
\mathrm{div}(V)&=\dfrac{\partial}{\partial x_{i_{\ell+1}}}\left(\left(\nabla^{j-\ell-1}\psi\right)_{i_{\ell+2}\ldots i_j}\left(\frac{1-|x|^2}2\right)^{N+j-\ell-1}\right)\\
&\quad+\left(\nabla^{j-\ell-1}\psi\right)_{i_{\ell+2}\ldots i_j}\left(\frac{1-|x|^2}2\right)^{N+j-\ell-1}\dfrac{\partial}{\partial x_{i_\ell+1}}\left(\log\sqrt{|g|}\right).
\end{align*}
Using that $\sqrt{|g|}=(\frac{2}{1-|x|^2})^N$, we have
\begin{align*}
\mathrm{div}(V)&=\dfrac{\partial\left(\nabla^{j-\ell-1}\psi\right)_{i_{\ell+2}\ldots i_j}}{\partial x_{i_{\ell+1}}}\left(\dfrac{1-|x|^2}2\right)^{N+j-\ell-1}\\
&\quad-(j-\ell-1)\left(\nabla^{j-\ell-1}\psi\right)_{i_{\ell+2}\ldots i_j}\left(\dfrac{1-|x|^2}2\right)^{N+j-\ell-2}x_{i_{\ell+1}}.
\end{align*}
By applying this and \eqref{eqdivtheo2} to \eqref{gkal}, and noting that $u$ has $j$ weak derivatives, which allows us to use \eqref{eqdivtheo2}, we obtain:
\begin{align*}
I_5&=\sum_{i_1\ldots i_j=1}^N\int_{B_R^{\mathbb H}}\dfrac{x_{i_1}\cdots x_{i_j}}{|x|^{N+j-1}}\left(\frac{1-|x|^2}2\right)^\ell\left(\nabla^\ell u\right)_{i_\ell\ldots i_1}\mathrm{div}(V)\mathrm dV_g\\
&=-\sum_{i_1\ldots i_j=1}^N\int_{B_R^{\mathbb H}}\left\langle\nabla\left(\dfrac{x_{i_1}\cdots x_{i_j}}{|x|^{N+j-1}}\left(\frac{1-|x|^2}2\right)^\ell\left(\nabla^\ell u\right)_{i_\ell\ldots i_1}\right),V\right\rangle_g\mathrm dV_g.
\end{align*}
By $\mathrm{div}(\frac{x}{|x|^N})=0$,
\begin{align*}
I_5&=-\sum_{i_1\ldots i_j=1}^N\int_{B_R^{\mathbb H}}\dfrac{x_{i_\ell+1}}{|x|^N}\dfrac{\partial}{\partial x_{i_{\ell+1}}}\left(\dfrac{x_{i_1}\cdots x_{i_\ell}x_{i_{\ell+2}}\cdots x_{i_j}}{|x|^{j-1}}\left(\frac{1-|x|^2}2\right)^\ell\left(\nabla^\ell u\right)_{i_\ell\ldots i_1}\right)\\
&\qquad\cdot\left(\frac{1-|x|^2}2\right)^{N+j-\ell-1}\left(\nabla^{j-\ell-1}\psi\right)_{i_{\ell+2}\ldots i_j}\mathrm dV_g.
\end{align*}
From $\sum_{i_{\ell+1}=1}^Nx_{i_{\ell+1}}\frac{\partial}{\partial x_{i_{\ell+1}}}(\frac{x_{i_1}\cdots x_{i_\ell}x_{i_{\ell+2}}\cdots x_{i_j}}{|x|^{j-1}})=0$, we get
\begin{align*}
I_5&=-\sum_{i_1\ldots i_j=1}^N\int_{B_R^{\mathbb H}}\dfrac{x_{i_1}\cdots x_{i_j}}{|x|^{N+j-1}}\dfrac{\partial}{\partial x_{i_{\ell+1}}}\left(\left(\frac{1-|x|^2}2\right)^\ell(\nabla^\ell u)_{i_{\ell}\ldots i_1}\right)\\
&\qquad\cdot\left(\frac{1-|x|^2}2\right)^{N+j-\ell-1}\left(\nabla^{j-\ell-1}\psi\right)_{i_{\ell+2}\ldots i_j}\mathrm dV_g\\
&=-\sum_{i_1\ldots i_j=1}^N\int_{B_R^{\mathbb H}}\dfrac{x_{i_1}\cdots x_{i_j}}{|x|^{N+j-1}}\left(\frac{1-|x|^2}2\right)^{N+j-1}\dfrac{\partial\left(\nabla^\ell u\right)_{i_\ell\ldots i_1}}{\partial x_{i_{\ell+1}}}\left(\nabla^{j-\ell-1}\psi\right)_{i_{\ell+2}\ldots i_j}\mathrm dV_g\\
    &\quad+\ell\sum_{i_1\ldots i_j=1}^N\int_{B_R^{\mathbb H}}\dfrac{x_{i_1}\cdots x_{i_j}}{|x|^{N+j-1}}\left(\frac{1-|x|^2}2\right)^{N+j-2}x_{i_{\ell+1}}\left(\nabla^\ell u\right)_{i_\ell\ldots i_1}\left(\nabla^{j-\ell-1}\psi\right)_{i_{\ell+2}\ldots i_j}\mathrm dV_g.
\end{align*}
Since
\begin{align*}
\sum_{i_1\ldots i_{\ell+1}=1}^N\left(\nabla^{\ell+1}u\right)_{i_{\ell+1}\ldots i_1}x_{i_1}\cdots x_{i_{\ell+1}}&=\sum_{i_1\ldots i_{\ell+1}=1}^N\dfrac{\partial\left(\nabla^\ell u\right)_{i_\ell\ldots i_1}}{\partial x_{i_{\ell+1}}}x_{i_1}\cdots x_{i_{\ell+1}}\\
&\quad -\ell|x|^2\frac{2}{1-|x|^2}\sum_{i_1\ldots i_\ell=1}^N\left(\nabla^\ell u\right)_{i_\ell\ldots i_1}x_{i_1}\cdots x_{i_{\ell}}
\end{align*}
which follows by \eqref{eqvgh51}, we are able to conclude
\begin{equation*}
I_5=-\sum_{i_1\ldots i_j=1}^N\int_{B_R^{\mathbb H}}\dfrac{x_{i_1}\cdots x_{i_j}}{|x|^{N+j-1}}\left(\frac{1-|x|^2}2\right)^{N+j-1}\left(\nabla^{\ell+1}u\right)_{i_{\ell+1}\ldots i_1}\left(\nabla^{j-\ell-1}\psi\right)_{i_{\ell+2}\ldots i_j}\mathrm dV_g.
\end{equation*}
This finishes the proof of \eqref{intpart} and, therefore, the proof of the lemma.
\end{proof}

\begin{proof}[Proof of Theorem \ref{theouv}] Let $(u_n)$ be a Cauchy sequence of radial functions in $\mathfrak C^{k,p}(B_R^{\mathbb H})$ with $u_n\to u$ in $L^p(B_R^{\mathbb H})$. Denote the sequence $(v_n)$ given by $v_n(d(x))=u_n(x)$ for all $x\in B_R^{\mathbb H}$. From Lemma \ref{lemmaagh}, the Cauchy-Schwarz inequality (on vectors in $\mathbb R^{N^j}$), and the equation \eqref{eqnormkcov},
\begin{align}
|v^{(j)}_n(d(x))|^2&=\left(\frac{1-|x|^2}2\right)^{2j}\left(\sum_{i_1\ldots i_j=1}^N(\nabla^j u_n)_{i_1\ldots i_j}\dfrac{x_{i_1}\cdots x_{i_j}}{|x|^j}\right)^2\nonumber\\
&\leq\left(\frac{1-|x|^2}2\right)^{2j}\left(\sum_{i_1\ldots i_j=1}^N(\nabla^j u_n)^2_{i_1\ldots i_j}\right)\left(\sum_{i_1\ldots i_j=1}^N\left(\dfrac{x_{i_1}\cdots x_{i_j}}{|x|^j}\right)^2\right)\nonumber\\
&=\left(\frac{1-|x|^2}2\right)^{2j}\left(\sum_{i_1\ldots i_j=1}^N(\nabla^j u_n)^2_{i_1\ldots i_j}\right)=|\nabla^ju_n|_g^2,\quad\forall j=0,\ldots,k.\label{eq3100}
\end{align}
Then $v_n\in L^p((0,R),\sinh^{N-1}(t))$ and
\begin{align*}
\|v_n^{(j)}-v_m^{(j)}\|_{L^p_{\sinh^{N-1}}}^p&=\int_0^R|v^{(j)}_n-v^{(j)}_m|^p\sinh^{N-1}(t)\mathrm dt\leq\dfrac{1}{\omega_{N-1}}\int_{B_R^{\mathbb H}}|\nabla u_n^j-\nabla u_m^j|_g^p\mathrm dV_g\\
&=\frac{1}{\omega_{N-1}}\|\nabla^ju_n-\nabla^ju_m\|^p_{L^p(B_R^{\mathbb H})},\quad\forall j=0,\ldots,k\mbox{ and }n,m\in \mathbb N,
\end{align*}
we used the same argument as in \eqref{eq3100} for $v_n-v_m$. Thus, $(v_n^{(j)})$ is a Cauchy sequence in $L^p((0,R),\sinh^{N-1}(t))$ for all $j=0,\ldots,k$ and, by Lemma \ref{lemmaagh1}, $v$ has $k$ weak derivatives with $v\in W^{k,p}((0,R),\sinh^{N-1}(t))$.

To prove \eqref{eqineuv}, we use \eqref{eq3100} and the fact that, up to subsequence, $v^{(j)}_n(t)$ and $\nabla^{j}u_n(x)$ converge almost everywhere to $v^{(j)}(t)$ and $\nabla^{j}u(x)$, respectively. Initially, we have $|\nabla^ku(x)|_g\geq|v^{(k)}(d(x))|$ almost everywhere in $B_R^{\mathbb H}\backslash\{0\}$. However, by applying the items (1) from both Theorem \ref{theo11} and Theorem \ref{theo12}, we expand this inequality to hold for all $x\in B_R^{\mathbb H}\backslash\{0\}$. Therefore, we concluded the proof of Theorem \ref{theouv}.
\end{proof}

\section{Relations between \texorpdfstring{$W^{k,p}_{\mathbb H,\mathrm{rad}}(B_R^{\mathbb H})$}{} and \texorpdfstring{$W^{k,p}((0,R),\sinh^{N-1}(t))$}{}}\label{sec4}

Throughout this section, unless specified otherwise, we assume $R\in(0,\infty]$. Our primary goal is to establish a comprehensive relationship between the conditions under which $u\in W^{k,p}_{\mathbb H,\mathrm{rad}}(B^{\mathbb H}_R)$ implies $v\in W^{k,p}((0,R),\sinh^{N-1}(t))$ and vice-versa, where $u$ and $v$ are related by $u(x)=v(d(x))$.

Proposition \ref{propa}, a direct consequence of Theorem \ref{theouv}, addresses one direction of this relationship. Proposition \ref{propc} will provide a counterexample demonstrating that $v\in W^{k,p}((0,R),\sinh^{N-1}(t))$ does not necessarily imply $u\in W^{k,p}_{\mathbb H,\mathrm{rad}}(B^{\mathbb H}_R)$.

The most challenging part of this section is the proof of Proposition \ref{propb}, which requires determining an expression for the derivatives of $u$ in terms of the derivatives of $v$, essentially the inverse of Lemma \ref{lemmaagh1}. Although an explicit equation was not obtained, we derived a recursive formula, expressed as a linear combination of terms similar to \eqref{termsderivative}, where each term arises from applying the recursive formula in \eqref{termsk+1} to $(\nabla u)_{i_k}$. This recursive formula was sufficient to complete the proof of Proposition \ref{propb}.

Theorem \ref{theo10} summarizes the results of this section. Its proof follows directly from Propositions \ref{propa}, \ref{propb}, and \ref{propc}.

\begin{prop}\label{propa}
For each $p\geq1$ and $k\geq1$ integer, we have
\begin{enumerate}
    \item[(a)] $W^{k,p}_{\mathbb H,\mathrm{rad}}(B_R^{\mathbb H})\hookrightarrow W^{k,p}((0,R),\sinh^{N-1}(t))$;
    \item[(b)] $W^{1,p}_{\mathbb H,\mathrm{rad}}(B_R^{\mathbb H})\equiv W^{1,p}((0,R),\sinh^{N-1}(t))$.
\end{enumerate}
\end{prop}

\begin{proof}
\textit{(a)} Using Theorem \ref{theouv} and \eqref{eqpolar}, we have
\begin{align*}
\|\nabla^ju\|^p_{L^p(B^{\mathbb H}_R)}&=\int_{B_R^{\mathbb H}}|\nabla^ju(x)|_g^p\mathrm dV_g\leq \omega_{N-1}\int_0^R|v^{(j)}(t)|^p\sinh^{N-1}(t)\mathrm dt\\
&=\omega_{N-1}\|v^{(j)}\|^p_{L^p_{\sinh^{N-1}}},\quad\forall j=0,\ldots,k.
\end{align*}

\noindent \textit{(b)} It similarly follows as the item \textit{(a)}, but using that in the Theorem \ref{theouv} the equality holds for $k=0$ and $k=1$. The equality $|\nabla^0u(x)|=|v^{(0)}(d(x))|$ is trivial and the case $k=1$ follows from
\begin{equation*}
\left(\nabla u\right)_i=\dfrac{x_i}{|x|}\dfrac{2}{1-|x|^2}v'(d(x))
\end{equation*}
and
\begin{equation*}
\left|\nabla u\right|_g^2=\left(\dfrac{1-|x|^2}2\right)^2\sum_{i=1}^N\left(\nabla u\right)_i^2=|v'(d(x))|^2.
\end{equation*}
\end{proof}

\begin{prop}\label{propb}
Assume $R\in(0,\infty)$. Let $k\geq2$ and $p\geq1$. If $N>(k-1)p$, then $W^{k,p}_{\mathbb H,\mathrm{rad}}(B_R^{\mathbb H})\equiv W^{k,p}((0,R),\sinh^{N-1}(t))$.
\end{prop}
\begin{proof}
Firstly, we claim that
\begin{equation}\label{dernorm}
    \left|\dfrac{\partial^k}{\partial x_{i_1}\cdots\partial x_{i_k}}|x|\right|\leq\dfrac{C_{N,k}}{|x|^{k-1}},\quad\forall i_1,\ldots,i_k=1,\ldots,N.
\end{equation}
Indeed, by induction on $k$, we can show that $\frac{\partial^k}{\partial x_{i_1}\cdots\partial x_{i_k}}|x|$ is composed of a linear combination of terms, which take the form:
\begin{equation*}
\dfrac{\delta_{a_1\widetilde a_1}\cdots\delta_{a_\ell\widetilde a_\ell}x_{\alpha_1}\cdots x_{\alpha_n}}{|x|^{2k-1-2\ell}},
\end{equation*}
where $2\ell+n=k$ and $a_1,\widetilde a_1,\ldots,a_\ell,\widetilde a_\ell,\alpha_1,\ldots,\alpha_n$ includes every possible permutations of $i_1,\ldots,i_k$. Therefore,
\begin{equation*}
\left|\dfrac{\partial^k}{\partial x_{i_1}\cdots\partial x_{i_k}}|x|\right|\leq\sum C\left|\dfrac{\delta_{a_1\widetilde a_1}\cdots\delta_{a_\ell\widetilde a_\ell}x_{j_1}\cdots x_{j_n}}{|x|^{2k-1-2\ell}}\right|\leq \sum\dfrac{C}{|x|^{2k-1-2\ell-n}}=\dfrac{C_{N,k}}{|x|^{k-1}},
\end{equation*}
which concludes \eqref{dernorm}.

Fix $v\in W^{k,p}((0,R),\sinh^{N-1}(t))$ and define $u(x)=v(d(x))$. We will prove, by induction on $k$, that $(\nabla^ku)_{i_1\ldots i_k}$ consists of a linear combination of terms, which these terms are expressed as:
\begin{equation}\label{termsderivative}
    \left(\dfrac{2}{1-|x|^2}\right)^{a_k}|x|^{b_k}PD_{k,\ell_k}|x|v^{(c_k)}(d(x)),
\end{equation}
where $1\leq a_k,c_k\leq k$, $0\leq b_k\leq k-1$, and $PD_{k,\ell_k}|x|$ denotes products of partial derivatives of $|x|$, where $k$ is the total number of partial derivatives and $\ell_k$ is the number of repeated derivatives. More specifically about $k$ and $\ell_k$, given a product of partial derivatives of $|x|$
\begin{equation*}
PD_{k,\ell_k}|x|=D^{\alpha_1}|x|\cdots D^{\alpha_j}|x|,
\end{equation*}
where $\alpha_1,\ldots,\alpha_j$ are multi-index with $|\alpha_1|,\ldots,|\alpha_j|\geq1$, we have that $k=|\alpha_1|+\cdots+|\alpha_j|$ and $\ell_k=\left(|\alpha_1|-1\right)+\cdots+\left(|\alpha_j|-1\right)$. Here the multi-index notation concerns the derivatives with index $i_1,\ldots,i_k$ which means that for a multi-index $\alpha=(\alpha^1,\ldots,\alpha^k)$, we define
\begin{equation*}
    D^\alpha|x|=\dfrac{\partial^{|\alpha|}}{\partial x_{i_1}^{\alpha^1}\cdots\partial x_{i_k}^{\alpha^k}}|x|
\end{equation*}

For $k=1$, we have that $(\nabla u)_{i_1}$ is a linear combination of terms as in \eqref{termsderivative} because
\begin{equation*}
(\nabla u)_{i_1}=\dfrac{\partial u}{\partial x_{i_1}}=\dfrac{2}{1-|x|^2}\dfrac{\partial|x|}{\partial x_{i_1}}v'(d(x)).
\end{equation*}
Refer to the recursive expression of the covariant derivative (see Equation \eqref{eqsf}):
\begin{align*}
    (\nabla^{k+1}u)_{i_1\ldots i_{k+1}}&=\dfrac{\partial (\nabla^{k}u)_{i_2\ldots i_{k+1}}}{\partial x_{i_1}}\\
    &\quad-\dfrac{2x_{i_1}}{1-|x|^2}(\nabla^{k}u)_{i_2\ldots i_{k+1}}-\dfrac{2}{1-|x|^2}\sum_{\ell=2}^kx_{i_\ell}(\nabla^{k}u)_{i_2\ldots i_{\ell-1}i_1i_{\ell+1}\ldots i_{k+1}}.
\end{align*}
Assuming $(\nabla^ku)_{i_1\ldots i_{k}}$ is given by a linear combination of terms as in \eqref{termsderivative}, the recursive expression of the covariant derivative guarantees that
\begin{equation*}
\dfrac{\partial}{\partial x_{i_j}}\left[\left(\dfrac{2}{1-|x|^2}\right)^{a_k}|x|^{b_k}PD_{k,\ell_k}|x|v^{(c_k)}(d(x))\right]
\end{equation*}
and
\begin{equation}\label{redundant}
    \dfrac{2}{1-|x|^2}|x|\dfrac{\partial|x|}{\partial x_{i_j}}\left[\left(\dfrac{2}{1-|x|^2}\right)^{a_k}|x|^{b_k}PD_{k,\ell_k}|x|v^{(c_k)}(d(x))\right]
\end{equation}
appear on the expression of $(\nabla^{k+1}u)_{i_1\ldots i_{k+1}}$. Therefore, each term \eqref{termsderivative} manifest with the following terms in $(\nabla^{k+1}u)_{i_1\ldots i_{k+1}}$
\begin{equation}\label{termsk+1}
\left\{\begin{array}{l}
\left(\dfrac{2}{1-|x|^2}\right)^{a_k+1}|x|^{b_k+1}PD_{k+1,\ell_k}|x|v^{(c_k)}(d(x)),\\
\left(\dfrac{2}{1-|x|^2}\right)^{a_k}|x|^{b_k-1}PD_{k+1,\ell_k}|x|v^{(c_k)}(d(x))\mbox{ if }b_k\geq1,\\
\left(\dfrac{2}{1-|x|^2}\right)^{a_k}|x|^{b_k}PD_{k+1,\ell_{k}+1}|x|v^{(c_k)}(d(x)),\\
\left(\dfrac{2}{1-|x|^2}\right)^{a_k+1}|x|^{b_k}PD_{k+1,\ell_k}|x|v^{(c_k+1)}(d(x)).
\end{array}\right.
\end{equation}
Note that the first term in \eqref{termsk+1} arises when we differentiate $(\frac{2}{1-|x|^2})^{a_k}$, and from \eqref{redundant}, the second term appears when we differentiate $|x|^{b_k}$. The third term results from differentiating $PD_{k,\ell_k}|x|$, and the fourth term comes from differentiating $v^{(c_k)}(d(x))$. Since every term in \eqref{termsk+1} is in the form of \eqref{termsderivative}, we conclude by induction that $(\nabla^ku)_{i_1\ldots i_k}$ is a linear combination of \eqref{termsderivative}.

Without loss of generality, it is sufficient to show that $|\nabla^ku|_g\in L^p(B_R^{\mathbb H})$ with $k\geq1$. Note that by \eqref{dernorm}, we have
\begin{equation*}
PD_{k,\ell_k}|x|\leq \dfrac{C}{|x|^{\ell_k}}.
\end{equation*}
Fix $j=1,\ldots,k$ and consider all the terms in $(\nabla^ku)_{i_1\ldots i_k}$ that contain $v^{(j)}(d(x))$. These terms arise from applying the recursive formula in \eqref{termsk+1} a total of $k-1$ times on $(\nabla u)_{i_k}$. Since they include $v^{(j)}(d(x))$, the fourth rule in \eqref{termsk+1} must have been applied $j-1$ times, and the other three rules were applied $k-j$ times in total. Given that $0<|x|\leq \overline R<1$, the worst-case estimate (with other cases providing a better estimate) occurs when only the third rule is applied. Thus,
\begin{equation*}
    (\nabla^ku)_{i_1\ldots i_k}\leq C\sum_{j=1}^k\dfrac{v^{(j)}(d(x))}{|x|^{k-j}}.
\end{equation*}
Consequently,
\begin{align*}
    \int_{B_R^{\mathbb H}}|\nabla^ku|_g^p\mathrm dV_g&\leq C\sum_{i_1\ldots i_k}\int_{B_R^{\mathbb H}}|(\nabla^ku)_{i_1\ldots i_k}|^p\mathrm dV_g\leq C\sum_{j=1}^k\int_0^R\left|\dfrac{v^{(j)}(t)}{\tanh^{k-j}(\frac{t}2)}\right|^p\sinh^{N-1}(t)\mathrm dt\\
    &\leq C\sum_{j=1}^k\int_0^R\left|\dfrac{v^{(j)}(t)}{\sinh^{k-j}(t)}\right|^p\sinh^{N-1}(t)\mathrm dt.
\end{align*}
Since $N>(k-1)p$, we can apply the Hardy inequality from Proposition \ref{prophardy} to obtain
\begin{equation*}
\int_{B_R^{\mathbb H}}|\nabla^ku|^p_g\mathrm dV_g\leq C\sum_{j=1}^k\int_0^R|v^{(j)}(t)|^p\sinh^{N-1}(t)\mathrm dt<\infty.
\end{equation*}
Therefore, $u\in W^{k,p}_{\mathbb H,\mathrm{rad}}(B_R^{\mathbb H})$, which completes the proof.
\end{proof}

\begin{prop}\label{propc}
Assume $R\in(0,\infty)$, $k\geq2$, and $N\leq (k-1)p$. Define the function $u\colon B_R^{\mathbb H}\to\mathbb R$ by
\begin{equation*}
u(x)=d(x),\quad\forall x\in B_R^{\mathbb H}.
\end{equation*}
Then $u\notin W^{k,p}_{\mathbb H,\mathrm{rad}}(B_R^{\mathbb H})$ with $v\in W^{k,p}((0,R),\sinh^{N-1}(t))$. In particular, this implies that
\begin{equation*}
W^{k,p}_{\mathbb H,\mathrm{rad}}(B_R^{\mathbb H})\not\equiv W^{k,p}((0,R),\sinh^{N-1}(t))\mbox{ if }N\leq(k-1)p.
\end{equation*}
\end{prop}
\begin{proof}
Observe that $v\in W^{k,p}((0,R),\sinh^{N-1}(t))$ since $v(t)=t$. For convenience, we fixed the coordinate $x_1$ in this proof, but any other coordinate $x_i$ could have been selected. Given $\varepsilon>\delta>0$, consider the conical annular sector given by
\begin{equation*}
\Omega_{\delta\varepsilon}=\{x\in B_R^{\mathbb H}\colon \delta|x|\leq x_1\leq\varepsilon|x|\}.
\end{equation*}

We claim that there exist constants $C_i>0$ and $\varepsilon>\delta>0$, depending only on $i$, such that
\begin{equation}\label{eqd1}
\left|\dfrac{\partial^i}{\partial x_1^i}(|x|)\right|\geq\dfrac{C_i}{|x|^{i-1}},\quad\forall x\in\Omega_{\delta\varepsilon}.
\end{equation}
Indeed, using induction on $i$, we can show that
\begin{equation}\label{eqd3}
\dfrac{\partial^i}{\partial x_1^i}(|x|)=\sum_{\ell=0}^{\lfloor\frac{i}2\rfloor}C_{\ell i}\dfrac{x_1^{i-2\ell}}{|x|^{2i-1-\ell}},
\end{equation}
where $C_{00}=1$ and the following recursive rules hold:
\begin{equation*}
\left\{\begin{array}{l}
     C_{0i+1}=-(2i-1)C_{0i},\\
      C_{\ell i+1}=(i-2\ell+2)C_{\ell-1i}-(2i-1-2\ell)C_{\ell i}\mbox{ for }\ell=1,\ldots,\lfloor\frac{i}2\rfloor,\\
      C_{\lfloor\frac{i}{2}\rfloor+1 i+1}=\left\{\begin{array}{ll}
           C_{\lfloor\frac{i}2\rfloor i},&\mbox{if }i\mbox{ odd},  \\
           0&\mbox{if }i\mbox{ even}. 
      \end{array}\right.
\end{array}\right.
\end{equation*}
Assuming $C_{\lfloor\frac{i}2\rfloor,i}>0$, we obtain from \eqref{eqd3} that, for $x\in\Omega_{\delta\varepsilon}$,
\begin{equation*}
\dfrac{\partial^i}{\partial x_1^i}(|x|)\geq C_{\lfloor\frac{i}{2}\rfloor i}\dfrac{\delta^{i-2\lfloor\frac{i}{2}\rfloor}}{|x|^{i-1}}+\sum_{\substack{\ell=0\\C_{\ell i}>0}}^{\lfloor\frac{i}{2}\rfloor-1}C_{\ell i}\dfrac{\delta^{i-2\ell}}{|x|^{i-1}}+\sum_{\substack{\ell=0\\ C_{\ell i}<0}}^{\lfloor\frac{i}{2}\rfloor-1}C_{\ell i}\dfrac{\varepsilon^{i-2\ell}}{|x|^{i-1}}.
\end{equation*}
By choosing $\varepsilon>\delta>0$ sufficiently small, we conclude that
\begin{equation*}
\dfrac{\partial^i}{\partial x_1^i}(|x|)\geq \dfrac{C_{\lfloor\frac{i}{2}\rfloor i}\delta^{i-2\lfloor\frac{i}{2}\rfloor}}2\dfrac{1}{|x|^{i-1}}>0,\quad\forall x\in\Omega_{\delta\varepsilon}.
\end{equation*}
Similarly, if $C_{\lfloor\frac{i}{2}\rfloor i}<0$, for $\varepsilon>\delta>0$ sufficiently small, we have
\begin{equation*}
\dfrac{\partial^i}{\partial x_1^i}(|x|)\leq \dfrac{C_{\lfloor\frac{i}{2}\rfloor i}\varepsilon^{i-2\lfloor\frac{i}{2}\rfloor}}2\dfrac{1}{|x|^{i-1}}<0,\quad\forall x\in\Omega_{\delta\varepsilon}.
\end{equation*}
This concludes the proof of \eqref{eqd1}.

Hereafter, we fix $\varepsilon,\delta>0$ such that \eqref{eqd1} holds for all $i=1,\ldots,k$. We claim that there exist constants $C>0$ and a small $\eta>0$ such that
\begin{equation}\label{eqd12}
\left|(\nabla^ku(x))_{1\cdots1}\right|\geq \dfrac{C}{|x|^{k-1}},\quad\forall x\in\Omega_{\delta\varepsilon}\cap B_\eta^{\mathbb R}.
\end{equation}
Recalling the methods from the proof of Proposition \ref{propb}, we find that $(\nabla^ku)_{1\cdots1}$ is a linear combination of terms like \eqref{termsderivative}, derived by iterating the changes in \eqref{termsk+1} $k-1$ times on $(\nabla u)_1=\frac{2}{1-|x|^2}\frac{\partial}{\partial x_1}(|x|)$. Since $v^{(2)}\equiv0$, the fourth term in \eqref{termsk+1} is always zero. The dominant growth term as $|x|\to0$ arises when only the third rule is applied, because
\begin{equation*}
\left|PD_{k,\ell_k}|x|\right|=\left|\dfrac{\partial^{\alpha_1}}{\partial x_1^{\alpha_1}}(|x|)\cdots\dfrac{\partial^{\alpha_j}}{\partial x_1^{\alpha_j}}(|x|)\right|\leq \dfrac{C}{|x|^{\ell_k}}.
\end{equation*}
More specifically, as $|x|\to0$, we obtain
\begin{equation*}
(\nabla^k u)_{1\cdots1}=\dfrac{2}{1-|x|^2}\dfrac{\partial^k}{\partial x_1^k}(|x|)+o\left(\dfrac{1}{|x|^{k-1}}\right).
\end{equation*}
Using \eqref{eqd1}, we choose a small $\eta>0$ such that \eqref{eqd12} holds, concluding the proof of \eqref{eqd12}.

Apllying \eqref{eqd12}, we then have
\begin{align}
    \int_{B_R^{\mathbb H}}|\nabla^ku(x)|_g^p\mathrm dV_g&=\int_{B_R^{\mathbb H}}|\nabla^ku(x)|_g^p\left(\dfrac{2}{1-|x|^2}\right)^N\mathrm dx\geq 2^N\int_{B_R^{\mathbb H}}|\nabla^ku(x)|_g^p\mathrm dx\nonumber\\
    &\geq 2^{N-k}\int_{B_{\overline R}^{\mathbb R}}\left|(\nabla^ku(x))_{1\cdots1}\right|^p\mathrm dx\nonumber\\
    &=2^{N-k}\int_0^{\overline R}\int_{\partial B_1^{\mathbb R}}\left|(\nabla^ku(ry))_{1\cdots1}\right|^pr^{N-1}\mathrm d\mathcal H^{N-1}(y)\mathrm dr\nonumber\\
    &\geq2^{N-k}\int_0^{\eta}\int_{\partial B_1^{\mathbb R}\cap\Omega_{\delta\varepsilon}}\left|(\nabla^ku(ry))_{1\cdots1}\right|^pr^{N-1}\mathrm d\mathcal H^{N-1}(y)\mathrm dr\nonumber\\
    &\geq2^{N-k}C\mathcal H^{N-1}\left(\partial B_1^{\mathbb R}\cap\Omega_{\delta\varepsilon}\right)\int_0^{\eta}r^{N-1+(1-k)p}\mathrm dr.\label{asjfnasfjka1}
\end{align}
Since $N\leq(k-1)p$, we have $N-1+(1-k)p\leq-1$, and thus $u\notin W^{k,p}_{\mathbb H,\mathrm{rad}}(B_R^{\mathbb H})$ by \eqref{asjfnasfjka1}. This concludes the proof of the proposition.
\end{proof}

\begin{proof}[Proof of Theorem \ref{theo10}]
The proof follows directly from Propositions \ref{propa}, \ref{propb}, and \ref{propc}.
\end{proof}

\section{Radial Lemmata and their consequences}\label{sec5}

Throughout this section, we always assume $R\in(0,\infty)$. In this section, our purpose is to establish radial lemmata for $u\in W^{k,p}_{\mathbb H,\mathrm{rad}}(B_R^{\mathbb H})$ and $v\in W^{k,p}((0,R),\sinh^{N-1}(t))$. Additionally, we provide proofs for Theorem \ref{theo11} and Corollary \ref{corcompact}, which ensure the continuous and compact embedding of $W^{k,p}_{\mathbb H,\mathrm{rad}}(B_R^{\mathbb H})$ into $L^q_{\sinh^{\theta}}(B_R^{\mathbb H})$.

Firstly, we derive the radial lemmata and a Hardy-type inequality for the Sobolev space with weight $\sinh^{N-1}(t)$ by applying inequalities established for the Sobolev space with weight $t^{N-1}$. In Proposition \ref{proprls1}, we improve the radial lemma specifically for $N=kp$.

Then, using these radial lemmata and a Hardy-type inequality (Proposition \ref{prophardy}), we prove Theorem \ref{theo11}. Finally, we establish Corollary \ref{corcompact} using the classical compact embedding and interpolation inequality.

Utilizing the inequalities
\begin{equation}\label{eqcompsinh}
t\leq \sinh(t)\leq \dfrac{\sinh(R)}{R}t,\quad\forall t\in[0,R],
\end{equation}
we conclude that
\begin{equation}\label{equivW}
W^{k,p}((0,R),\sinh^{N-1}(t))=W^{k,p}((0,R),t^{N-1}).
\end{equation}
As the study of radial lemmata for functions in the space on the right-hand side is well established in Section 2 of \cite{zbMATH05978504}, we can immediately derive the following lemma and propositions from those results.

\begin{lemma}\label{lemma31}
There exists $C=C(N,p,R)>0$ such that
\begin{equation*}
|v(R)|\leq C\|v\|_{W^{1,p}_{\sinh^{N-1}}},\quad\forall v\in W^{1,p}((0,R),\sinh^{N-1}(t)).
\end{equation*}
\end{lemma}

\begin{prop}\label{prop31}
Let $v\in W^{k,p}((0,R),\sinh^{N-1}(t))$ with $1\leq p<\infty$ and $R\in(0,\infty]$. Then there exists $V\in AC_{loc}^{k-1}((0,R])$ such that
\begin{equation*}
    v=V \quad \text{a.e. on } \quad (0,R).
\end{equation*}
Moreover, $V^{(k)}$ (in the classical sense) exists a.e. on $(0,R)$ and $V^{(k)}$ is a measurable function and
\begin{equation*}
\int_0^R\left|V^{(j)}(s)\right|^p\sinh^{N-1}(s)\mathrm ds<\infty,\quad\forall j=0,1,\ldots,k.
\end{equation*}
\end{prop}

\begin{prop}\label{proprls}
Assume $N>kp$. Then there exists $C=C(N,k,p,R)>0$ such that for all $v\in W^{k,p}((0,R),\sinh^{N-1}(t))$ it holds
\begin{equation}\label{rlsv}
|v(t)|\leq C\dfrac{1}{\sinh^{\frac{N-kp}{p}}(t)}\|v\|_{W^{k,p}_{\sinh^{N-1}}},\quad\forall t\in(0,R].
\end{equation}
\end{prop}

\begin{prop}\label{propworst}
Assume $N=kp$ and $p>1$. Then there exists a constant $C=C(k,p,R)>0$ such that for all $v\in W^{k,p}((0,R),\sinh^{N-1}(t))$ and $t\in(0,R]$ it holds
\begin{equation}
|v(t)|\leq \left(\log\frac{R}{t}\right)^{\frac{p-1}p}\dfrac{\|v^{(k)}\|_{L^p_{\sinh^{N-1}}}}{(k-1)!}+C\|v\|_{W^{k,p}_{\sinh^{N-1}}}.
\end{equation}
\end{prop}

\begin{prop}\label{asfjanfkslns}
If $N=kp$ and $p=1$, then $W^{k,p}((0,R),\sinh^{N-1}(t))\hookrightarrow C([0,R])$. In other words, $W^{k,1}((0,R),\sinh^{k-1}(t))\hookrightarrow C([0,R])$.
\end{prop}

\begin{prop}\label{prophardy}
Given $j=0,\ldots,k$ with $N>jp$, there exists $C_j=C(j,p,R,k,N)>0$ such that for all $v\in W^{k,p}((0,R),\sinh^{N-1}(t))$,
\begin{equation}\label{tl3}
\int_0^R\left|\dfrac{v^{(k-j)}(s)}{\sinh^j(s)}\right|^p\sinh^{N-1}(s)\mathrm ds\leq C_j\sum_{i=k-j}^k\int_0^R|v^{(i)}(s)|^p\sinh^{N-1}(s)\mathrm ds.
\end{equation}
In particular, assuming $N>kp$,
\begin{equation*}
\left\|\dfrac{v}{\sinh^k(s)}\right\|_{L^p_{\sinh^{N-1}}(0,R)}\leq C_k\|v\|_{W^{k,p}_{\sinh^{N-1}}}.
\end{equation*}
\end{prop}

In the Sobolev limit case ($N=kp$), it will be necessary to improve the estimate in Proposition \ref{propworst}. We achieve a better estimate in our next proposition because
\begin{equation*}
\log\left(\frac{R}{t}\right)>\log\dfrac{\tanh\left(\frac{R}2\right)}{\tanh\left(\frac{t}2\right)},\quad\forall t\in(0,R).
\end{equation*}

\begin{prop}\label{proprls1}
Assume $N=kp$ and $p>1$. Then there exists a constant $C=C(k,p,R)>0$ such that for all $v\in W^{k,p}((0,R),\sinh^{N-1}(t))$ and $t\in(0,R]$ it holds
\begin{equation}\label{LR-A}
|v(t)|\leq \left(\log\dfrac{\tanh(\frac{R}2)}{\tanh(\frac{t}2)}\right)^{\frac{p-1}p}\dfrac{\|v^{(k)}\|_{L^p_{\sinh^{N-1}}}}{(k-1)!}+C\|v\|_{W^{k,p}_{\sinh^{N-1}}}.
\end{equation}
Moreover, if $k=1$ and $v(R)=0$, then \eqref{LR-A} holds with $C=0$.
\end{prop}
\begin{proof}
Firstly, using induction on $i=1,\ldots,k-1$, we are able to prove that
\begin{equation}\label{sc1}
\dfrac{\mathrm d^i}{\mathrm ds^{i}}\left(\sinh^{k-1}(s)\right)=\sum_{j=0}^iC_{ij}\sinh^{k-1-j}(s)\cosh^j(s),\quad\forall i=1,\ldots,k-1,
\end{equation}
where the constants $C_{ij}=C_{ij}(k)\in\mathbb R$ satisfy the following recursive expressions for each $i=1,\ldots,k-2$ 
\begin{equation*}
\left\{\begin{array}{l}
C_{10}=0,\ C_{11}=k-1, \\
C_{i+10}=C_{i1},\ C_{i+1i}=(k-i)C_{ii-1},\ C_{i+1i+1}=(k-1-i)C_{ii},\\
C_{i+1j}=C_{ij-1}(k-j)+C_{ij+1}(j+1),\quad \forall j=1,\ldots,i-1.
\end{array}\right.
\end{equation*}
Note that $C_{k-1k-1}=(k-1)!$. Applying $(k-1)$-times the integration by parts, we obtain
\begin{align}
\int_t^Rv^{(k)}(s)\sinh^{k-1}(s)\mathrm ds&=(-1)^{k-1}\int_t^Rv'(s)\dfrac{\mathrm d^{k-1}}{\mathrm ds^{k-1}}\left(\sinh^{k-1}(s)\right)\mathrm ds\nonumber\\
&\quad+(-1)^{k-1}\sum_{j=1}^{k-1}(-1)^j\left.\left[v^{(j)}(s)\dfrac{\mathrm d^{k-1-j}}{\mathrm ds^{k-1-j}}\left(\sinh^{k-1}(s)\right)\right]\right|_t^R.\label{sc2}
\end{align}
On the other hand, using that $(1+x)^{\frac{k-1}2}=1+O(x)$ (as $x\to0$), we have
\begin{align*}
\int_t^Rv'(s)\dfrac{\mathrm d^{k-1}}{\mathrm ds^{k-1}}\left(\sinh^{k-1}(s)\right)\mathrm ds&=\sum_{j=0}^{k-2}C_{k-1j}\int_t^Rv'(s)\sinh^{k-1-j}(s)\cosh^j(s)\mathrm ds\\
&\quad+C_{k-1k-1}\int_t^Rv'(s)\left(1+\sinh^2(s)\right)^{\frac{k-1}2}\mathrm ds\\
&=\sum_{j=0}^{k-2}C_{k-1j}\int_t^Rv'(s)\sinh^{k-1-j}(s)\cosh^j(s)\mathrm ds\\
&\quad+\!(k-1)!(v(R)-v(t))\!+\!(k-1)!\!\int_t^R\!v'(s)O(\sinh^2(s))\mathrm ds.
\end{align*}
Using \eqref{sc2} and isolating $v(t)$ we obtain
\begin{align}
v(t)&=v(R)+\int_t^Rv'(s)O(\sinh^2(s))\mathrm ds+\dfrac{(-1)^k}{(k-1)!}\int_t^Rv^{(k)}(s)\sinh^{k-1}(s)\mathrm ds\nonumber\\
&\quad+\dfrac{1}{(k-1)!}\sum_{j=1}^{k-1}(-1)^j\left.\left[v^{(j)}(s)\dfrac{\mathrm d^{k-1-j}}{\mathrm ds^{k-1-j}}\left(\sinh^{k-1}(s)\right)\right]\right|_t^R.\label{sc3}
\end{align}
Since $N=kp$ and $v^{(j)}\in W^{k-j,p}((0,R),\sinh^{N-1}(t))$ for all $j=1,\ldots,k-1$, we can apply the Proposition \ref{proprls} to get $|v^{(j)}(t)|\sinh^j(t)\leq C\|v\|_{W^{k,p}_{\sinh^{N-1}}}$ for all $t\in(0,R]$. Then,
\begin{equation}\label{sc4}
\int_t^Rv'(s)O(\sinh^2(s))\mathrm ds\leq C\|v\|_{W^{k,p}_{\sinh^{N-1}}}
\end{equation}
and, by \eqref{sc1},
\begin{equation}\label{sc5}
|v^{(j)}(t)|\dfrac{\mathrm d^{k-1-j}}{\mathrm dt^{k-1-j}}\left(\sinh^{k-1}(t)\right)\leq C\|v\|_{W^{k,p}_{\sinh^{N-1}}}\sum_{i=0}^{k-1-j}\sinh^{k-1-i-j}(t)\cosh^i(t)\leq C\|v\|_{W^{k,p}_{\sinh^{N-1}}},
\end{equation}
for all $j=1,\ldots,k-1$ and $t\in(0,R]$. Using the estimates \eqref{sc4} and \eqref{sc5} together with the Lemma \ref{lemma31} on \eqref{sc3}, we conclude
\begin{align}
|v(t)|&\leq\dfrac{1}{(k-1)!}\int_t^Rv^{(k)}(s)\sinh^{k-1}(s)\mathrm ds+C\|v\|_{W^{k,p}_{\sinh^{N-1}}}\nonumber\\
&\leq\dfrac{\|v^{(k)}\|_{L^p_{\sinh^{N-1}}}}{(k-1)!}\left(\int_t^R\sinh^{-1}(s)\mathrm ds\right)^{\frac{p-1}{p}}+C\|v\|_{W^{k,p}_{\sinh^{N-1}}}\nonumber\\
&=\left(\log\dfrac{\tanh(\frac{R}2)}{\tanh(\frac{t}2)}\right)^{\frac{p-1}p}\dfrac{\|v^{(k)}\|_{L^p_{\sinh^{N-1}}}}{(k-1)!}+C\|v\|_{W^{k,p}_{\sinh^{N-1}}},\label{sc6}
\end{align}
where $\sinh^{-1}(s)$ denotes $\left(\sinh(s)\right)^{-1}$ and not $\mathrm{arsinh}(s)$ the inverse function of the hyperbolic sine. Therefore we concluded \eqref{LR-A}.
\end{proof}
\begin{cor}\label{corrlnkp}
Assume $N=kp$ and $p>1$. Then there exists a constant $C=C(k,p,R)>0$ such that for all $u\in W^{k,p}_{\mathbb H,\mathrm{rad}}(B_R^{\mathbb H})$ and $x\in \overline{B_R}\backslash\{0\}$ it holds
\begin{equation*}
|u(x)|\leq \dfrac{\|\nabla^ku\|_{L^p(B_R^{\mathbb H})}}{\omega_{N-1}^{\frac1p}(k-1)!}\left(\log\frac{\overline R}{|x|}\right)^{\frac{p-1}p}+C\|u\|_{W^{k,p}(B_R^{\mathbb H})}.
\end{equation*}
\end{cor}

\begin{proof}[Proof of Theorem \ref{theo11}]
\textit{(1)} This item is the Proposition \ref{prop31}.

\vspace{0.2cm}

\noindent \textit{(2)} The equation \eqref{rls} follows directly from the radial lemma \eqref{rlsv} and Theorem \ref{theouv} with $t=\log\frac{1+|x|}{1-|x|}=d(x)$. Using \eqref{rls} and $\theta\geq0$, we have
\begin{equation}\label{ccsa}
|u|^{\frac{\theta p}{N-kp}}\sinh^\theta d(x)\leq C\|u\|^{\frac{\theta p}{N-kp}}_{W^{k,p}(B_R^{\mathbb H})}.
\end{equation}
Therefore from \eqref{ccsa} and the classical embedding $W^{k,p}_{\mathbb H}(B_R^{\mathbb H})\hookrightarrow L^{\frac{Np}{N-kp}}(B_R^{\mathbb H})$ in hyperbolic space, we conclude
\begin{align*}
\int_{B_R^{\mathbb H}}|u|^{\frac{p(\theta+N)}{N-kp}}\sinh^\theta d(x)\mathrm dV_g&\leq C\|u\|^{\frac{\theta p}{N-kp}}_{W^{k,p}(B_R^{\mathbb H})}\|u\|^{\frac{Np}{N-kp}}_{L^{\frac{Np}{N-kp}}(B_R^{\mathbb H})}\\
&\leq C\|u\|_{W^{k,p}(B_R^{\mathbb H})}^{\frac{p(\theta+N)}{N-kp}},
\end{align*}
which implies the continuous embedding $W^{k,p}_{\mathbb H,\mathrm{rad}}(B_R^{\mathbb H})\hookrightarrow L^{\frac{p(\theta+N)}{N-kp}}_{\sinh^\theta}(B_R^{\mathbb H})$.

\vspace{0.2cm}

\noindent\textit{(3)} The equation \eqref{rlsc} is a consequence of the radial lemma \eqref{LR-A} and Theorem \ref{theouv} with $t=\log\frac{1+|x|}{1-|x|}$. Therefore the embedding follows from \eqref{rlsc} and the claim:
\begin{equation}\label{claimsa}
\left(\log\frac{R}{t}\right)^{\frac{p-1}p}+1\leq C_\varepsilon t^{-\varepsilon},\quad\forall t\in(0,R],\varepsilon>0,
\end{equation}
for some $C_\varepsilon>0$ which does not depend on $t$. To prove \eqref{claimsa}, note that
\begin{align*}
\mbox{\eqref{claimsa} holds}&\Leftrightarrow\lim_{t\to0}\dfrac{\left(\log\frac{R}{t}\right)^{\frac{p-1}p}+1}{t^{-\varepsilon}}=0\quad\forall\varepsilon>0\\
&\Leftrightarrow\lim_{t\to0}\dfrac{\log R-\log t}{t^{-\varepsilon}}=0\quad\forall\varepsilon>0.
\end{align*}
Therefore, applying L'H\^opital's rule, we conclude the claim and the proof of \textit{(3)}.

\vspace{0.2cm}

\noindent\textit{(4)} Follows directly from Proposition \ref{asfjanfkslns} and Theorem \ref{theouv}.
\end{proof}

\begin{proof}[Proof of Corollary \ref{corcompact}]
Let $(u_n)$ be a bounded sequence in $W^{k,p}_{\mathbb H,\mathrm{rad}}(B_R^{\mathbb H})$. Since $W^{k,p}_{\mathbb H}(B_R^{\mathbb H})\hookrightarrow L^1(B_R^{\mathbb H})$ is compact (see \cite[Theorem 10.1]{zbMATH01447265}), we have that, up to a subsequence, $(u_n)$ is Cauchy in $L^1(B_R^{\mathbb H})$. Remember that $R=\log\frac{1+\overline R}{1-\overline R}$. By interpolation inequality and $1\leq q<p^*_\theta=\frac{(\theta+N)p}{N-kp}$, there exists $\alpha\in(0,1]$ such that
\begin{align*}
\left\|\sinh^{\frac{\theta}q}d(x)u_n-\sinh^{\frac{\theta}{q}}d(x)u_m\right\|_{L^q(B_R^{\mathbb H})}&\leq\left\|\sinh^{\frac{\theta}q}d(x)u_n-\sinh^{\frac{\theta}{q}}d(x)u_m\right\|_{L^1(B_R^{\mathbb H})}^\alpha\\
&\quad\cdot\left\|\sinh^{\frac{\theta}q}d(x)u_n-\sinh^{\frac{\theta}{q}}d(x)u_m\right\|_{L^{p^*_\theta}(B_R^{\mathbb H})}^{1-\alpha}.
\end{align*}
Diving everything by $\sinh^{\frac{\theta}{q}}(R)$ and using that $\left(\frac{\sinh d(x)}{\sinh(R)}\right)^{\frac{\theta}{q}}\leq\left(\frac{\sinh d(x)}{\sinh(R)}\right)^{\frac{\theta}{p^*_\theta}}\leq1$ for all $x\in B_R^{\mathbb H}$, we have
\begin{align*}
\dfrac{\left\|u_n-u_m\right\|_{L^q_{\sinh^\theta}(B_R^{\mathbb H})}}{\sinh^{\frac{\theta}{q}}(R)}&\leq\left\|\dfrac{\sinh^{\frac{\theta}q}d(x)}{\sinh^{\frac{\theta}q}(R)}u_n-\dfrac{\sinh^{\frac{\theta}q}d(x)}{\sinh^{\frac{\theta}q}(R)}u_m\right\|^\alpha_{L^1(B_R^{\mathbb H})}\\
&\quad\cdot\left\|\dfrac{\sinh^{\frac{\theta}{p^*_\theta}}d(x)}{\sinh^{\frac{\theta}{p^*_\theta}}(R)}u_n-\dfrac{\sinh^{\frac{\theta}{p^*_\theta}}d(x)}{\sinh^{\frac{\theta}{p^*_\theta}}(R)}u_m\right\|^{1-\alpha}_{L^{p^*_\theta}(B_R^{\mathbb H})}\\
&\leq\|u_n-u_m\|^\alpha_{L^1(B_R^{\mathbb H})}\dfrac{\left\|\sinh^{\frac{\theta}{p^*_\theta}}d(x)u_n-\sinh^{\frac{\theta}{p^*_\theta}}d(x)u_m\right\|^{1-\alpha}_{L^{p^*_\theta}(B_R^{\mathbb H})}}{\sinh^{\frac{\theta}{p^*_\theta}}(R)}\\
&\leq C\|u_n-u_m\|^\alpha_{L^1(B_R^{\mathbb H})}\|u_n-u_m\|^{1-\alpha}_{W^{k,p}(B_R^{\mathbb H})},
\end{align*}
where we used the continuous embedding $W^{k,p}_{\mathbb H,\mathrm{rad}}(B_R^{\mathbb H})\hookrightarrow L^{p^*_\theta}_{\sinh^\theta}(B_R^{\mathbb H})$. Therefore $(u_n)$ is Cauchy in $L^q_{\sinh^\theta}(B_R^{\mathbb H})$ and hence it converges.
\end{proof}

An analogous theorem can be obtained for the weighted Sobolev space $W^{k,p}((0,R),\sinh^{N-1}(t))$. The proof is very similar to the proof of Theorem \ref{theo11} and will be left to the reader. Alternatively, by setting $\alpha_0=\cdots=\alpha_k=N-1$ in \cite[Theorem 1.1]{arXiv:2302.02262} and considering the equivalence \eqref{equivW}, we derive the following theorem.

\begin{theo}\label{theoweighted}
Assume $R\in(0,\infty)$, $N\geq kp$, $\theta\geq N-kp-1$, $p\geq1$ real numbers, and $k\geq1$ integer.
\begin{flushleft}
$\mathrm{(1)}$ Every function $v\in W^{k,p}((0,R),\sinh^{N-1}(t))$ is almost everywhere equal to a function $V$ in $C^{k-1}((0,R])$. Additionally, all derivatives of $V$ of order $k$ (in the classical sense) exist almost everywhere for $t\in(0,R)$.\\
$\mathrm{(2)}$ If $N>kp$, then $W^{k,p}((0,R),\sinh^{N-1}(t))$ is continuously embedded in $L^q_{\sinh^\theta}(0,R)$ for every $1\leq q\leq\frac{(\theta+1)p}{N-kp}$.\\
$\mathrm{(3)}$ If $N=kp$, then $W^{k,p}((0,R),\sinh^{N-1}(t))$ is compactly embedded in $L^q_{\sinh^\theta}(0,R)$ for all $1\leq q<\infty$.
\end{flushleft}
\end{theo}

\section{Decay Lemma and its consequences}\label{sec6}

In this section, we consider the case $R=\infty$. Our objective is to develop a decay lemma (Lemma \ref{decaylemma}) for $u\in W^{1,p}_{\mathbb H,\mathrm{rad}}(B_R^{\mathbb H})$ and establish conditions for the embedding $W^{k,p}_{\mathbb H,\mathrm{rad}}(\mathbb H^N)\hookrightarrow L^q_{\sinh^\theta}(\mathbb H^N)$. Following the same arguments as in the Euclidean scenario, developing an asymptotic decay for functions in $W^{1,p}_{\mathbb H,\mathrm{rad}}(\mathbb H^N)$ is crucial for proving these embeddings, which will be demonstrated in the proof of Theorem \ref{theo12}.

At the end of this section, we will present similar results for the weighted Sobolev space $W^{k,p}((0,\infty),\sinh^{N-1}(t))$, noting that their proofs follow analogous arguments to those used for the Sobolev space $W^{k,p}_{\mathbb H,\mathrm{rad}}(\mathbb H^N)$.

\begin{lemma}\label{decaylemma}
For each $u\in W^{1,p}_{\mathbb H,\mathrm{rad}}(\mathbb H^N)$, it holds
\begin{equation*}
|u(x)|\leq \left(\dfrac{p}{\omega_{N-1}}\right)^{\frac1p}\|u\|_{L^p(\mathbb H^N)}^{\frac{p-1}p}\|\nabla u\|_{L^p(\mathbb H^N)}^{\frac1p}\sinh^{\frac{1-N}{p}}d(x),\quad\forall x\in\mathbb H^N\backslash\{0\}.
\end{equation*}
\end{lemma}
\begin{proof}
It is enough to prove the lemma for radial functions in $C^{\infty}_0(\mathbb H^N)$ because each $u\in W^{1,p}_{\mathbb H,\mathrm{rad}}(\mathbb H^N)$ can be approximated by radial functions in $C^{\infty}_{0}(\mathbb H^N)$ in the norm of $W^{1,p}_{\mathbb H}(\mathbb H^N)$. For now, assume $p>1$. Using that $u(x)=v(t)$ with $t=d(x)$, we have
\begin{align*}
|u(x)|^p&=|v(t)|^p=-\int_t^\infty(|v(s)|^p)'\mathrm ds=-p\int_t^\infty|v(s)|^{p-2}v(s)v'(s)\mathrm ds\\
&\leq p\int_t^\infty|v(s)|^{p-1}|v'(s)|\left(\frac{\sinh(s)}{\sinh(t)}\right)^{N-1}\mathrm ds\leq p\|v\|_{L^p_{\sinh^{N-1}}}^{p-1}\|v'\|_{L^p_{\sinh^{N-1}}}\sinh^{1-N}(t)\\
&=\frac{p}{\omega_{N-1}}\|u\|_{L^p(\mathbb H^N)}^{p-1}\|\nabla u\|_{L^p(\mathbb H^N)}\sinh^{1-N}d(x).
\end{align*}
This concludes the case $p>1$.

The case $p=1$ immediately follows from
\begin{equation*}
|u(x)|=|v(t)|=\left|\int_t^\infty v'(s)\left(\frac{\sinh(s)}{\sinh(t)}\right)^{N-1}\mathrm ds\right|\leq \frac{1}{\omega_{N-1}}\|\nabla u\|_{L^1(\mathbb H^N)}\sinh^{1-N}d(x).
\end{equation*}
Therefore, we conclude the lemma.
\end{proof}

Before proving Theorem \ref{theo12}, we need to establish the following proposition, which is the case $k=1$ of the theorem.

\begin{prop}\label{proppf12}
Assume $\theta\geq0$ and $p\geq1$ real numbers.
\begin{flushleft}
    $\mathrm{(a)}$ If $N>p$, then the following continuous embedding holds:
    \begin{equation*}
    W^{1,p}_{\mathbb H,\mathrm{rad}}(\mathbb H^N)\hookrightarrow L^q_{\sinh^\theta}(\mathbb H^N)\quad\text{if}\quad p\leq q\leq p^*,
    \end{equation*}
    where
    \begin{equation*}
    p^*:=\dfrac{(\theta+N)p}{N-p}.
    \end{equation*}
    Moreover, it is a compact embedding if one of the following two conditions is fulfilled:
    \begin{flushleft}
    $\mathrm{(i)}$ $\theta=N-1$ and $p<q<p^*$;\\
    $\mathrm{(ii)}$ $\theta<N-1$ and $p\leq q<p^*$.
    \end{flushleft}
    $\mathrm{(b)}$ If $N=p$, then the following continuous embedding holds:
    \begin{equation*}
    W^{1,p}_{\mathbb H,\mathrm{rad}}(\mathbb H^N)\hookrightarrow L^q_{\sinh^\theta}(\mathbb H^N)\quad\text{if}\quad p\leq q<\infty.
    \end{equation*}
    Moreover, it is compact embedding if one of the following two conditions is fulfilled:
    \begin{flushleft}
    $\mathrm{(i)}$ $\theta=N-1$ and $q>p$;\\
    $\mathrm{(ii)}$ $\theta<N-1$ and $q\geq p$.
    \end{flushleft}
\end{flushleft}
\end{prop}
\begin{proof}
    Let us prove both the continuous embeddings first. Let $q\geq p$ and $R\in(0,\infty)$ be fixed. Given $x\in\mathbb H^N$ with $d(x)\geq R$, by the Lemma \ref{decaylemma} we have
    \begin{equation*}
|u(x)|^{q-p}\leq\dfrac{C}{\sinh^{\frac{(N-1)(q-p)}{p}} d(x)}\|u\|_{W^{1,p}(\mathbb H^N)}^{q-p}\leq\dfrac{C}{\sinh^{\frac{(N-1)(q-p)}{p}}(R)}\|u\|_{W^{1,p}(\mathbb H^N)}^{q-p}.
    \end{equation*}
    Thus,
    \begin{align}
\int_{\mathbb H^N\backslash B_R^{\mathbb H}}\!|u|^q\sinh^\theta\! d(x)\mathrm dV_g\!&\leq\! \dfrac{C}{\sinh^{\frac{(N-1)(q-p)}p}(R)}\|u\|^{q-p}_{W^{1,p}(\mathbb H^N)}\int_{\mathbb H^N\backslash B_R^{\mathbb H}}|u|^{p}\sinh^\theta d(x)\mathrm dV_g\nonumber\\
&\leq\! \dfrac{C}{\sinh^{\frac{(N-1)(q-p)}p+N-1-\theta}\!(R)}\|u\|^{q-p}_{W^{1,p}(\mathbb H^N)}\!\int_{\mathbb H^N\backslash B_R^{\mathbb H}}\!|u|^{p}\sinh^{N-1}\!d(x)\mathrm dV_g.\label{rews}
    \end{align}
    On the other hand, Theorem \ref{theo11} guarantees that $\|u\|_{L^q_{\sinh^\theta}(B_R^{\mathbb H})}\leq C\|u\|_{W^{1,p}(\mathbb H^N)}$ and from \eqref{rews} we conclude the continuous embeddings.

    Now we need to prove the compact embeddings. Let us show that if $u_n\rightharpoonup0$ in $W^{1,p}_{\mathbb H,\mathrm{rad}}(\mathbb H^N)$, then $u_n\rightarrow0$ in $L^q_{\sinh^\theta}(\mathbb H^N)$. From our hypothesis for compact embeddings, we have $\frac{(N-1)(q-p)}p+N-1-\theta>0$. Then, by \eqref{rews} and $(u_n)$ is bounded in $W^{1,p}_{\mathbb H,\mathrm{rad}}(\mathbb H^N)$, given $\varepsilon>0$ there exists $R_0>0$ such that
    \begin{equation*}
    \int_{\mathbb H^N\backslash B_{R_0}^{\mathbb H}}|u_n|^q\sinh^\theta d(x)\mathrm dV_g<\dfrac{\varepsilon}2,\quad\forall n\in\mathbb N.
    \end{equation*}
    On the other hand, the compact embedding of Theorem \ref{theo11} guarantees $n_0\in\mathbb N$ such that
    \begin{equation*}
        \int_{B_{R_0}^{\mathbb H}}|u_n|^q\sinh^\theta d(x)\mathrm dV_g<\dfrac{\varepsilon}2,\quad\forall n\geq n_0.
    \end{equation*}
    Therefore $u_n\to0$ in $L^q_{\sinh^\theta}(\mathbb H^N)$ which concludes the proof of the proposition.
\end{proof}

\begin{proof}[Proof of Theorem \ref{theo12}]
The proof for $k=1$ is the proof of Proposition \ref{proppf12}. Assume the theorem holds for $k-1$; we will prove it for $k$. Since $N>(k-1)p$ and $u,u'\in W^{k-1,p}_{\mathbb H,\mathrm{rad}}(\mathbb H^N)$, we have $u\in W^{1,q}_{\mathbb H,\mathrm{rad}}(\mathbb H^N)$ for $p\leq q\leq\frac{Np}{N-(k-1)p}$. Denote $\overline p=\frac{Np}{N-(k-1)p}$. In other words, the following embedding is continuous
    \begin{equation}\label{eq58}
    W^{k,p}_{\mathbb H,\mathrm{rad}}(\mathbb H^N)\hookrightarrow W^{1,\overline p}_{\mathbb H,\mathrm{rad}}(\mathbb H^N).
    \end{equation}
    Since $\frac{(\theta+N)\overline p}{N-\overline p}=\frac{(\theta+N)p}{N-kp}$, applying the Proposition \ref{proppf12} we have that
    \begin{equation}\label{eq59}
    W^{1,\overline p}_{\mathbb H,\mathrm{rad}}(\mathbb H^N)\hookrightarrow L_{\sinh^\theta}^q(\mathbb H^N)
    \end{equation}
    where $p\leq q\leq p^*$ in Sobolev case and $p\leq q<\infty$ in Sobolev Limit case. \eqref{eq58} together with \eqref{eq59} conclude the continuous embeddings.

    Under the conditions (i) or (ii), we obtain that \eqref{eq59} is a compact embedding by Proposition \ref{proppf12} in both Sobolev and Sobolev Limit cases. Therefore, \eqref{eq58} and \eqref{eq59} imply the desired compact embedding. This concludes the proof of the theorem.
    \end{proof}

    \begin{lemma}
For each $v\in W^{1,p}((0,\infty),\sinh^{N-1}(t))$, it holds
\begin{equation*}
|v(t)|\leq p^{\frac1p}\|v\|_{L^p_{\sinh^{N-1}}}^{\frac{p-1}p}\|v'\|_{L^p_{\sinh^{N-1}}}^{\frac1p}\sinh^{\frac{1-N}{p}}(t),\quad\forall t\in(0,\infty).
\end{equation*}
\end{lemma}
\begin{proof}
It follows by employing a similar approach to the proof of Lemma \ref{decaylemma}.
\end{proof}

\begin{theo}\label{theo122}
Assume $\theta\geq N-kp-1$, $p\geq1$ real nunmbers, and $k\geq1$ integer.
\begin{flushleft}
$\mathrm{(1)}$ \justifying{Every function $v\in W^{k,p}((0,\infty),\sinh^{N-1}(t))$ is almost everywhere equal to a function $V$ in $C^{k-1}(0,\infty)$. In addition, all derivatives of $V$ of order $k$ (in the classical sense) exist almost everywhere for $t\in(0,\infty)$.}

\noindent$\mathrm{(2)}$ If $N>kp$, then the following continuous embedding holds:
\begin{equation*}
W^{k,p}((0,\infty),\sinh^{N-1}(t))\hookrightarrow L^q_{\sinh^\theta}(0,\infty)\quad\text{if}\quad p\leq q\leq p^*,
\end{equation*}
where
\begin{equation*}
    p^*=\dfrac{(\theta+1)p}{N-kp}.
\end{equation*}
Moreover, it is a compact embedding if one of the following two conditions is fulfilled:
\begin{flushleft}
$\mathrm{(i)}$ $\theta=N-1$ and $p<q<p^*$;

\noindent$\mathrm{(ii)}$ $\theta<N-1$ and $p\leq q<p^*$.
\end{flushleft}
\noindent$\mathrm{(3)}$ \justifying{If $N=kp$, then the following continuous embedding holds:}
\begin{equation*}
W^{k,p}((0,\infty),\sinh^{N-1}(t))\hookrightarrow L^q_{\sinh^\theta}(0,\infty)\quad\text{if}\quad p\leq q<\infty.
\end{equation*}
Moreover, it is compact embedding if one of the following two conditions is fulfilled:
\begin{flushleft}
\noindent$\mathrm{(i)}$ $\theta=N-1$ and $q>p$;

\noindent$\mathrm{(ii)}$ $\theta<N-1$ and $q\geq p$.
\end{flushleft}
\end{flushleft}
\end{theo}
\begin{proof}
Using the same ideas as in the proof of Proposition \ref{proppf12} and Theorem \ref{theo12} but using Theorem \ref{theoweighted} instead of Theorem \ref{theo11}.
\end{proof}

\section{Adams-type inequality}\label{sec7}

In this section, we assume $R\in(0,\infty)$ and $N=kp$. Our primary goal is to establish the proof of Theorem \ref{theo16} and \ref{theo15}, which provide some necessary and sufficient conditions for which $\mu>0$ ensures the finiteness of the suprema in \eqref{e33} and \eqref{e34}.

We begin with the proof of Theorem \ref{theo16}. First, we need to verify that $\int_{0}^R\exp(\mu|v|^{\frac{p}{p-1}})\sinh^\theta(t)\mathrm dt$ is finite for every $v\in W^{k,p}((0,R),\sinh^{N-1}(t))$. This is the item $(a)$ of Theorem \ref{theo16} and is concluded by Proposition \ref{prop50}. Next, we use Proposition \ref{prop51} to prove Proposition \ref{propitema}, which together correspond to the item $(b)$ of Theorem \ref{theo16}. Following that, item $(c)$ of Theorem \ref{theo16} is established in Proposition \ref{prop54}. Finally, we conclude this section with the proof of Theorem \ref{theo15}.

\begin{prop}\label{prop50}
Let $\theta>-1$ and $W^{k,p}((0,R),\sinh^{N-1}(t))$ with $N=kp$ and $p>1$. Then for any $\mu\geq0$ we have $\exp(\mu|v|^{\frac{p}{p-1}})\in L^1_{\sinh^\theta}(0,R)$ for all $v\in W^{k,p}((0,R),\sinh^{N-1}(t))$.
\end{prop}
\begin{proof}
Let $v\in W^{k,p}((0,R),\sinh^{N-1}(t))$ and $\theta>-1$. We claim that
\begin{equation}\label{eqq34}
\lim_{r\to0}\dfrac{v(r)}{\left|\log \tanh(\frac{r}{2})\right|^{\frac{p-1}{p}}}=0.
\end{equation}
Indeed, given $\varepsilon>0$ there exists $R(\varepsilon)>0$ such that $\|v^{(k)}\|_{L^p(0,R(\varepsilon))}<\varepsilon$. Using \eqref{LR-A} (with $R=R(\varepsilon)$) we have \eqref{eqq34}. Fixed $\delta=(\frac{\theta+1}{2\mu})^{\frac{p-1}{p}}$, by \eqref{eqq34}, there exists $\varepsilon>0$ such that $|v(r)|\leq\delta|\log\tanh(\frac{r}{2})|^{\frac{p-1}p}$ for all $r<\varepsilon$. Then
\begin{equation*}
\int_0^\varepsilon e^{\mu|v|^{\frac{p}{p-1}}}\sinh^{\theta}(t)\mathrm dt\leq \int_0^\varepsilon\dfrac{\sinh^\theta(t)}{\tanh^{\frac{\theta+1}2}(\frac{t}{2})}\mathrm dt\leq\left(\cosh\varepsilon+1\right)^{\frac{\theta+1}2}\int_{0}^{\varepsilon}\sinh^{\frac{\theta-1}2}(t)\mathrm dt
\end{equation*}
is finite because $\theta>-1$. This concludes $\exp(\mu|v|^{\frac{p}{p-1}})\in L^1_{\sinh^\theta}(0,R)$ since Proposition \ref{proprls1} estimates (by a constant depending on $\varepsilon$) the term $\int_{\varepsilon}^Re^{\mu|v|^{\frac{p}{p-1}}}\sinh^\theta(t)\mathrm dt$. Therefore, we finished the proof.
\end{proof}

\begin{prop}\label{prop51}
Let $\theta>-1$ and the Sobolev space $W^{k,p}((0,R),\sinh^{N-1}(t))$ with $N=kp$ and $p>1$. If $\mu<\mu_0$, then $\mathcal{L}_{\mu}$ is finite.
\end{prop}
\begin{proof}
Our proof is based on the approach taken in \cite[Theorem 2.6]{zbMATH05372150}. We aim in find $\mu_0$ such that for any $\mu<\mu_0$, the condition $\mathcal{L}_{\mu}<\infty$ holds. To achieve this, we first consider the following inequality:
\begin{equation}\label{provisorio}
(a+b)^q\leq (1+\varepsilon)a^q+C_q\left(1+\dfrac1{\varepsilon^{q-1}}\right)b^q,\quad\forall a,b\geq0 \mbox{ and }\varepsilon>0,
\end{equation}
where $C_{q}>0$ depends only on $q\geq1$.

Let $v\in W^{k,p}((0,R),\sinh^{N-1}(t))$ with $\|v\|_{W^{k,p}_{\sinh^{N-1}}}\leq1$. By the Proposition \ref{proprls1},
\begin{equation*}
|v(t)|\leq\dfrac{1}{(k-1)!}\left(\log\frac{\overline{R}}{\tanh(\frac{t}{2})}\right)^{\frac{p-1}p}+C,\quad\forall t\in (0,R].
\end{equation*}
Using \eqref{provisorio} we conclude
\begin{equation*}
|v(t)|^\frac{p}{p-1}\leq \frac{1+\varepsilon}{\left[(k-1)!\right]^{\frac{p}{p-1}}}\log\frac{\overline R}{\tanh(\frac{t}{2})}+C_{\varepsilon,p},\quad\forall t\in (0,R].
\end{equation*}
Setting 
\begin{equation*}
A_\varepsilon=\dfrac{1+\varepsilon}{\left[(k-1)!\right]^{\frac{p}{p-1}}}\mu,
\end{equation*}
we have
\begin{equation*}
\int_{0}^Re^{\mu|v|^{\frac{p}{p-1}}}\sinh^\theta (t)\mathrm dt\leq C\int_{0}^R\dfrac{\sinh^\theta(t)}{\tanh^{A_\varepsilon}(\frac{t}{2})}\mathrm dt\leq C\int_0^R\sinh^{\theta-A_\varepsilon}(t)\mathrm dt.
\end{equation*}
Note that the last integral is finite if and only if 
\begin{equation*}
\theta-\dfrac{1+\varepsilon}{[(k-1)!]^{\frac{p}{p-1}}}\mu>-1.
\end{equation*}
Since 
\begin{equation*}
    \mu<\mu_0=(\theta+1)\left[(k-1)!\right]^{\frac{p}{p-1}},
\end{equation*}
we get the desired result by taking $\varepsilon>0$ sufficiently small. This concludes the proof.
\end{proof}

\begin{prop}\label{propitema}
Let $W^{k,p}((0,R),\sinh^{N-1}(t))$ be the weighted Sobolev space with $N=kp$ and $p>1$. For every $0\leq\mu<\mu_0$, $\mathcal{L}_{\mu}$ is attained. Moreover, if $v$ is a maximizer for $\mathcal{L}_{\mu}$, then $v$ can be chosen nonnegative and such that $\|v\|_{W^{k,p}_{\sinh^{N-1}}}=1$.
\end{prop}
\begin{proof}
Let $(v_n)$ in $W^{k,p}((0,R),\sinh^{N-1}(t))$ be a maximizing sequence for $\mathcal{L}_\mu$ such that $\|u_n\|_{W^{k,p}(B_R^{\mathbb H})}\leq1$, see equation \eqref{e34}. By Theorem \ref{theoweighted}, there exists a function $v_0\in W^{k,p}((0,R),\sinh^{N-1}(t))$ such that, up to a subsequence, $v_n\rightharpoonup v_0$ in $W^{k,p}((0,R),\sinh^{N-1}(t))$ and $v_n\rightarrow v_0$ in $L_{\sinh^\theta}^q(0,R)$ for all $q\geq1$. From the inequality $|e^x-e^y|\leq|x-y|(e^x+e^y)$ we get
\begin{align}
\Bigg|\int_{0}^R\Big(e^{\mu|v_n|^{\frac{p}{p-1}}}&-e^{\mu|v_0|^{\frac{p}{p-1}}}\Big)\sinh^\theta (t)\mathrm dt\Bigg|\nonumber\\
&\leq\mu\int_{0}^R\left||v_n|^{\frac{p}{p-1}}-|v_0|^{\frac{p}{p-1}}\right|\left(e^{\mu|v_n|^{\frac{p}{p-1}}}+e^{\mu|v_0|^{\frac{p}{p-1}}}\right)\sinh^{\theta}(t)\mathrm dt.\label{eqtx}
\end{align}
Fixing $q>1$ with $q\mu<\mu_0$ we have

\begin{align}
    \int_{0}^R&\left||v_n|^{\frac{p}{p-1}}-|v_0|^{\frac{p}{p-1}}\right|e^{\mu|v_n|^{\frac{p}{p-1}}}\sinh^\theta(t)\mathrm dt\nonumber\\
    &\leq \left(\int_{0}^R\left||v_n|^{\frac{p}{p-1}}-|v_0|^{\frac{p}{p-1}}\right|^{\frac{q}{q-1}}\sinh^\theta(t)\mathrm dt\right)^{\frac{q-1}{q}}\left(\int_{0}^Re^{q\mu|v_n|^{\frac{p}{p-1}}}\sinh^{\theta}(t)\mathrm dt\right)^{\frac1q}\nonumber\\
&\leq\|v_n-v_0\|_{L^{\frac{p}{p-1}\frac{q}{q-1}}_{\sinh^\theta}(0,R)}^{\frac{p}{p-1}}\left(\int_{0}^Re^{q\mu|v_n|^{\frac{p}{p-1}}}\sinh^\theta(t)\mathrm dt\right)^{\frac1q}.\label{eqt10}
\end{align}
By Proposition \ref{prop51}, the integral on the right-hand side is uniformly bounded in $n$. Using \eqref{eqtx} and \eqref{eqt10} with $v_n\to v_0$ in $L^{\frac{p}{p-1}\frac{q}{q-1}}_{\sinh^\theta}(0,R)$ we conclude
\begin{align*}
\mathcal{L}_\mu&=\sup_{\substack{v\in W^{k,p}((0,R),\sinh^{N-1}(t)) \\ \|v\|_{W^{k,p}_{\sinh^{N-1}}}\leq1}}\int_{0}^Re^{\mu|v|^{\frac{p}{p-1}}}\sinh^\theta(t)\mathrm dt=\lim\int_{0}^Re^{\mu|v_n|^{\frac{p}{p-1}}}\sinh^\theta(t)\mathrm dt\\
&=\int_{0}^Re^{\mu|v_0|^{\frac{p}{p-1}}}\sinh^\theta(t)\mathrm dt.
\end{align*}

To prove that $\|v_0\|_{W^{k,p}_{\sinh^{N-1}}}=1$, assume $\|v_0\|_{W^{k,p}_{\sinh^{N-1}}}<1$ and define $\widetilde v=v_0/\|v_0\|_{W^{k,p}_{\sinh^{N-1}}}$ to get a contradiction with the fact that $v_0$ attains the supremum. We can assume that $v_0$ is nonnegative by replacing $v_0$ by $|v_0|$ if necessary. This concludes the proof.
\end{proof}

\begin{prop}\label{prop54}
Let $W^{k,p}((0,R),\sinh^{N-1}(t))$ be the Sobolev space with $N=kp$, $R\in(0,\infty)$, and $p>1$. If $\theta>-1$, then
\begin{equation*}
    \sup_{\substack{v\in W^{k,p}((0,R),\sinh^{N-1}(t)) \\ \|v\|_{W^{k,p}_{\sinh^{N-1}}}\leq1}}\int_{0}^Re^{\mu|v|^{\frac{p}{p-1}}}\sinh^\theta(t)\mathrm dt=\infty,\quad\forall \mu>\mu_0.
\end{equation*}
\end{prop}
\begin{proof}
Our proof is based on \cite{zbMATH04099653}, \cite[Proposition 3.4]{arXiv:2302.02262}, and \cite[Theorem 1.5]{zbMATH07827601}. Let $\phi\in C^\infty[0,1]$ be such that
\begin{equation*}
\left\{\begin{array}{ll}
\phi(0)=\phi'(0)=\cdots=\phi^{(k+1)}(0)=0,\ \phi'\geq0,\\
\phi(1)=\phi'(1)=1,\ \phi''(1)=\cdots=\phi^{(k-1)}(1)=0.
\end{array}\right.
\end{equation*}
Consider $0<\varepsilon<\frac12$ and
\begin{equation*}
H(t)=\left\{\begin{array}{llll}
     \varepsilon\phi\left(\dfrac{t}{\varepsilon}\right),&\mbox{if }0<t\leq\varepsilon,  \\
     t,&\mbox{if }\varepsilon<t\leq1-\varepsilon,\\
     1-\varepsilon\phi\left(\dfrac{1-t}{\varepsilon}\right),&\mbox{if }1-\varepsilon<t\leq1,\\
     1,&\mbox{if }t>1.
\end{array}\right.
\end{equation*}

Let $n\in\mathbb N$ and let the sequence $(\psi_{n,\varepsilon})_n$ given by
\begin{equation*}
\psi_{n,\varepsilon}(t)=H\left(\dfrac{\log\frac{R}{t}}{\log n}\right),\quad t\in(0,R).
\end{equation*}
It is not hard to see that $\psi_{n,\varepsilon}\in W^{k,p}_{0}((0,R),\sinh^{N-1}(t))$. We claim that
\begin{equation}\label{eqa0}
\|\psi_{n,\varepsilon}\|^p_{W^{k,p}_{\sinh^{N-1}}}\leq[(k-1)!]^p(\log n)^{1-p}\left[1+2^{p}\varepsilon\|\phi'\|^p_\infty+O\left((\log n)^{-1}\right)\right].
\end{equation}
To prove this, we make the following three claims, which together establish \eqref{eqa0}:
\begin{equation}\label{eqa1}
\|\psi_{n,\varepsilon}\|^p_{L^p_{\sinh^{N-1}}}=O\left((\log n)^{-p}\right),
\end{equation}
\begin{equation}\label{eqa2}
\|\psi^{(i)}_{n,\varepsilon}\|_{L^p_{\sinh^{N-1}}}^p=O\left((\log n)^{-p}\right)\quad\forall i=1,\ldots,k-1,
\end{equation}
and
\begin{equation}\label{eqa3}
\|\psi^{(k)}_{n,\varepsilon}\|^p_{L^p_{\sinh^{N-1}}}\leq[(k-1)!]^p(\log n)^{1-p}\left[1+2^{p}\varepsilon\|\phi'\|^p_\infty+O\left((\log n)^{-1}\right)\right].
\end{equation}

Let us prove \eqref{eqa1}. Doing the change of variables $s=(\log\frac{R}{t})/\log n$, we have
\begin{align}
\|\psi_{n,\varepsilon}\|^p_{L^p_{\sinh^{N-1}}}&=\int_0^R|\psi_{n,\varepsilon}|^p\sinh^{N-1}(t)\mathrm dt=\int_0^R\left|H\left(\dfrac{\log\frac{R}{t}}{\log n}\right)\right|^p\sinh^{N-1}(t)\mathrm dt\nonumber\\
&=R\log n\int_0^\infty|H(s)|^{p}\sinh^{N-1}(Rn^{-s})n^{-s}\mathrm ds\nonumber\\
&=R\log n\Bigg[\!\int_0^\varepsilon\!\left|\varepsilon\phi\!\left(\frac{s}{\varepsilon}\right)\right|^p\!\sinh^{N-1}(Rn^{-s})n^{-s}\mathrm ds\!+\!\int_{\varepsilon}^{1-\varepsilon}\!s^p\sinh^{N-1}(Rn^{-s})n^{-s}\!\mathrm ds\nonumber\\
&\quad+\!\int_{1-\varepsilon}^1\!\left|1\!-\!\varepsilon\phi\left(\frac{1-s}{\varepsilon}\right)\right|^p\!\sinh^{N-1}\!(Rn^{-s})n^{-s}\!\mathrm ds\!+\!\int_1^\infty\sinh^{N-1}(Rn^{-s})n^{-s}\mathrm ds\Bigg]\nonumber\\
&\leq R\log n\Bigg[\int_0^\varepsilon\left|\varepsilon\phi\left(\frac{s}{\varepsilon}\right)\right|^p\sinh^{N-1}(Rn^{-s})n^{-s}\mathrm ds+\int_\varepsilon^\infty\sinh^{N-1}(Rn^{-s})n^{-s}\mathrm ds\Bigg]\nonumber\\
&\leq R\log n\left[\int_0^\varepsilon\left|\varepsilon\phi\left(\frac{s}{\varepsilon}\right)\right|^p\sinh^{N-1}(Rn^{-s})n^{-s}\mathrm ds+\sinh^{N-1}(Rn^{-\varepsilon})\frac{n^{-\varepsilon}}{\log n}\right]\nonumber\\
&=R\log n\int_0^\varepsilon\left|\varepsilon\phi\left(\frac{s}{\varepsilon}\right)\right|^p\sinh^{N-1}(Rn^{-s})n^{-s}\mathrm ds+O\left((\log n)^{-p}\right).\label{eqa4}
\end{align}
Fix $\gamma\in(\frac{1}{k+1},1)$. Since $\sinh^{N-1}(Rn^{-s})\leq\sinh^{N-1}(R)n^{-(N-1)s}$ for all $s>0$ and $\phi(t)\leq t^{k+1}$ for $t>0$ sufficiently small, we have
\begin{align*}
I&:=\int_0^\varepsilon\left|\varepsilon\phi\left(\frac{s}{\varepsilon}\right)\right|^p\sinh^{N-1}(Rn^{-s})n^{-s}\mathrm ds\\
&\leq\sinh^{N-1}(R)\left[\int_0^{\varepsilon(\log n)^{-\gamma}}\left|\varepsilon\phi\left(\frac{s}{\varepsilon}\right)\right|^pn^{-Ns}\mathrm ds+\int^\varepsilon_{\varepsilon(\log n)^{-\gamma}}\left|\varepsilon\phi\left(\frac{s}{\varepsilon}\right)\right|^pn^{-Ns}\mathrm ds\right]\\
&\leq \sinh^{N-1}(R)\varepsilon^p\left[\left|\phi\left((\log n)^{-\gamma})\right)\right|^p\int_0^{\varepsilon(\log n)^{-\gamma}}n^{-Ns}\mathrm ds+\int_{\varepsilon(\log n)^{-\gamma}}^\varepsilon n^{-Ns}\mathrm ds\right]\\
&\leq\sinh^{N-1}(R)\varepsilon^p\left[(\log n)^{-\gamma p(k+1)}\frac{1}{N\log n}+\dfrac{n^{-N\varepsilon(\log n)^{-\gamma}}}{N\log n}\right]\\
&=O\left((\log n)^{-p-1}\right).
\end{align*}
Applying this on \eqref{eqa4}, we conclude \eqref{eqa1}.

Let us prove \eqref{eqa2}. By induction $i=1,\ldots,k$, we have
\begin{equation*}
\psi^{(i)}_{n,\varepsilon}(t)=\dfrac{1}{t^i}\sum_{j=1}^i\dfrac{c(j,i)}{(\log n)^j}H^{(j)}\left(\dfrac{\log\frac{R}{t}}{\log n}\right),
\end{equation*}
where $c(j,i)$ satisfies $c(1,1)=-1$, $c(1,i)=(-1)^i(i-1)!$, $c(i,i)=(-1)^i$ and $c(j,i+1)=-ic(j,i)-c(j-1,i)$ for each $j=2,\ldots,i$. Then
\begin{equation}
\psi_{n,\varepsilon}^{(i)}(t)=\dfrac1{t^i}\dfrac{(-1)^i(i-1)!}{\log n}H'\left(\dfrac{\log\frac{R}{r}}{\log n}\right)+O\left((\log n)^{-2}\right),\quad\forall i=1,\ldots,k.
\end{equation}
Since $\frac{\sinh^{N-1}(t)}{t^{N}}\leq \frac{1}{t}+C(N,R)$ for all $t\leq R$, we obtain, for any $i=1,\ldots, k$,
\begin{align}
\|\psi^{(i)}_{n,\varepsilon}\|^p_{L^p_{\sinh^{N-1}}}&\leq\left(\dfrac{(i-1)!}{\log n}\right)^p\int_0^R\left|H'\left(\dfrac{\log\frac{R}{t}}{\log n}\right)+O\left((\log n)^{-1}\right)\right|^p\dfrac{\sinh^{N-1}(t)}{t^{ip}}\mathrm dt\nonumber\\
&\leq\left(\dfrac{(i-1)!}{\log n}\right)^p\!\int_0^R\!\left|H'\left(\dfrac{\log\frac{R}{t}}{\log n}\right)\!+\!O\left((\log n)^{-1}\right)\right|^pt^{N-ip-1}\mathrm dt+O\left((\log n)^{-p}\right).\label{eqa7}
\end{align}
Using the definition of $H$ and that $N>ip$ for all $i=1,\ldots,k-1$, we have

\begin{align*}
\|\psi^{(i)}_{n,\varepsilon}\|^p_{L^p_{\sinh^{N-1}}}&\leq\left(\dfrac{(i-1)!}{\log n}\right)^p\Bigg[\left|\|\phi'\|_\infty+O\left((\log n)^{-1}\right)\right|^p\int_{Rn^{-1}}^{Rn^{\varepsilon-1}}t^{N-ip-1}\mathrm dt\\
&\quad+\int_{Rn^{\varepsilon-1}}^{Rn^{-\varepsilon}}t^{N-ip-1}\mathrm dt+\left|\|\phi'\|_\infty+O\left((\log n)^{-1}\right)\right|^p\int_{Rn^{-\varepsilon}}^{R}t^{N-ip-1}\mathrm dt\Bigg]\\
&\quad+O\left((\log n)^{-p}\right)\\
&\leq \left(\dfrac{(i-1)!}{\log n}\right)^p\dfrac{R^{N-ip}}{N-ip}\Big[\left|\|\phi'\|_\infty+O\left((\log n)^{-1}\right)\right|^pn^{-(1-\varepsilon)(N-ip)}\\
&\quad+n^{-\varepsilon(N-ip)}+\left|\|\phi'\|_\infty+O\left((\log n)^{-1}\right)\right|^p\Big]+O\left((\log n)^{-p}\right)\\
&=O\left((\log n)^{-p}\right).
\end{align*}
This concludes the proof of \eqref{eqa2}.

To establish \eqref{eqa0}, it is enough to prove \eqref{eqa2}. By \eqref{eqa7} for $i=k$, we get
\begin{align*}
    \|\psi^{(k)}_{n,\varepsilon}\|_{L^p_{\sinh^{N-1}}}^p&=\left(\dfrac{(k-1)!}{\log n}\right)^p\Bigg[\left|\|\phi'\|_\infty+O\left((\log n)^{-1}\right)\right|^p\int_{Rn^{-1}}^{Rn^{\varepsilon-1}}t^{-1}\mathrm dt\\
&\quad+\int_{Rn^{\varepsilon-1}}^{Rn^{-\varepsilon}}t^{-1}\mathrm dt+\left|\|\phi'\|_\infty+O\left((\log n)^{-1}\right)\right|^p\int_{Rn^{-\varepsilon}}^{R}t^{-1}\mathrm dt\Bigg]+O\left((\log n)^{-p}\right)\\
    &=\left(\dfrac{(k-1)!}{\log n}\right)^p\Big[\left|\|\phi'\|_\infty+O\left((\log n)^{-1}\right)\right|^p\varepsilon\log n+\log n\\
    &\quad+\left|\|\phi'\|_\infty+O\left((\log n)^{-1}\right)\right|^p\varepsilon\log n\Big]+O\left((\log n)^{-p}\right)\\
    &=(k-1)!(\log n)^{1-p}\left[1+2\varepsilon\left|\|\phi'\|_\infty+O\left((\log n)^{-1}\right)\right|^p\right]+O\left((\log n)^{-p}\right)\\
    &\leq(k-1)!(\log n)^{1-p}\left[1+2^p\varepsilon\|\phi'\|_\infty^p+O\left((\log n)^{-1}\right)\right].
\end{align*}
This completes the proof of the inequalities \eqref{eqa2} and \eqref{eqa0} as desired.

Set
\begin{equation*}
    v_{n,\varepsilon}(t)=\dfrac{\psi_{n,\varepsilon}(t)}{\|\psi_{n,\varepsilon}\|_{W^{k,p}_{\sinh^{N-1}}}}.
\end{equation*}
Given
\begin{equation*}
    \mu>\mu_0=(\theta+1)\left[(k-1)!\right]^{\frac{p}{p-1}},
\end{equation*}
inequality \eqref{eqa0} guarantees that

\begin{align*}
\int_{0}^Re^{\mu|v_{n,\varepsilon
}|^{\frac{p}{p-1}}}&\sinh^\theta(t)\mathrm dt\geq \exp\left(\frac{\mu}{\|\psi_{n,\varepsilon}\|^{\frac{p}{p-1}}_{W^{k,p}_{\sinh^{N-1}}}}\right)\int_0^{\frac{R}{n}}\sinh^\theta(t)\mathrm dt\\
&\geq\exp\left(\frac{\mu}{\|\psi_{n,\varepsilon}\|^{\frac{p}{p-1}}_{W^{k,p}_{\sinh^{N-1}}}}\right)\int_0^{\frac{R}{n}}t^\theta\mathrm dt\\
&\geq C\exp\left(\dfrac{\mu\log n}{\left[(k-1)!\right]^{\frac{p}{p-1}}\left[1+2^p\varepsilon\|\phi'\|_\infty+O\left((\log n)^{-1}\right)\right]^{\frac{1}{p-1}}}-(\theta+1)\log n\right).
\end{align*}
Therefore, taking $\varepsilon>0$ sufficiently small, we have that the term on the right-hand side tends to be infinite when $n\to\infty$. This concludes the proof.
\end{proof}

\begin{proof}[Proof of Theorem \ref{theo16}]
It is a direct consequence of Propositions \ref{prop50}, \ref{prop51}, \ref{propitema}, and \ref{prop54}.
\end{proof}

One approach to prove Theorem \ref{theo15} is to adapt the proof of Theorem \ref{theo16} to this specific case. However, we will provide an alternative proof that uses Theorem \ref{theo16}. By utilizing the equivalence between the norms of the spaces $W^{k,p}_{\mathbb H,\mathrm{rad}}(B^{\mathbb H}_R)$ and $W^{k,p}((0,R),\sinh^{N-1}(t))$, assuming $N=kp$, as given by item (3) of Theorem \ref{theo10}, we have the following relation:
\begin{equation}\label{eqequiv}
\omega_{N-1}^{\frac1p}\|v\|_{W^{k,p}_{\sinh^{N-1}}}\leq\|u\|_{W^{k,p}(B_R^{\mathbb H})}\leq C_0\|v\|_{W^{k,p}_{\sinh^{N-1}}},
\end{equation}
where $C_0>0$ is a constant. Before using this relation, we need to separately prove attainability, as it does not follow directly from Theorem \ref{theo16}. This requires the compact embedding from Theorem \ref{theo11}, so we assume $\theta\geq0$. Consequently, we prove the following proposition, whose proof is analogous to Proposition \ref{propitema}. For completeness, we have included the proof.

\begin{prop}\label{propitemaS}
Let $\theta\geq0$ and $W^{k,p}_{\mathbb H,\mathrm{rad}}(B_R^{\mathbb H})$ be the Sobolev space with $N=kp$ and $p>1$. For every $0\leq\mu<\mu_0$, $\mathcal{L}_{\mu,\mathbb H}$ is attained. Moreover, if $u$ is a maximizer for $\mathcal{L}_{\mu,\mathbb H}$, then $u$ can be chosen nonnegative and such that $\|u\|_{W^{k,p}(B_R^{\mathbb H})}=1$.
\end{prop}
\begin{proof}
Let $(u_n)$ in $W^{k,p}_{\mathbb H,\mathrm{rad}}(B_R^{\mathbb H})$ be a maximizing sequence for $\mathcal{L}_\mu$ with $\|u_n\|_{W^{k,p}(B_R^{\mathbb H})}\leq1$. By Theorem \ref{theo11} there exists $u_0\in W^{k,p}_{\mathbb H,\mathrm{rad}}(B_R^{\mathbb H})$ such that, up to a subsequence, $u_n\rightharpoonup u_0$ in $W^{k,p}_{\mathbb H,\mathrm{rad}}(B_R^{\mathbb H})$ and $u_n\rightarrow u_0$ in $L_{\sinh^\theta}^q(B_R^{\mathbb H})$ for all $q\geq1$. From $|e^x-e^y|\leq|x-y|(e^x+e^y)$ we get
\begin{align}
\Bigg|\int_{B_R^{\mathbb H}}\Big(e^{\mu|u_n|^{\frac{p}{p-1}}}&-e^{\mu|u_0|^{\frac{p}{p-1}}}\Big)\sinh^\theta d(x)\mathrm dV_g\Bigg|\nonumber\\
&\leq\mu\int_{B_R^{\mathbb H}}\left||u_n|^{\frac{p}{p-1}}-|u_0|^{\frac{p}{p-1}}\right|\left(e^{\mu|u_n|^{\frac{p}{p-1}}}+e^{\mu|u_0|^{\frac{p}{p-1}}}\right)\sinh^{\theta}d(x)\mathrm dV_g.\label{eqtxS}
\end{align}
Fixing $q>1$ with $q\mu<\mu_0$ we have

\begin{align}
    \int_{B_R^{\mathbb H}}&\left||u_n|^{\frac{p}{p-1}}-|u_0|^{\frac{p}{p-1}}\right|e^{\mu|u_n|^{\frac{p}{p-1}}}\sinh^\theta d(x)\mathrm dV_g\nonumber\\
    &\leq \left(\int_{B_R^{\mathbb H}}\left||u_n|^{\frac{p}{p-1}}-|u_0|^{\frac{p}{p-1}}\right|^{\frac{q}{q-1}}\sinh^\theta d(x)\mathrm dV_g\right)^{\frac{q-1}{q}}\left(\int_{B_R^{\mathbb H}}e^{q\mu|u_n|^{\frac{p}{p-1}}}\sinh^{\theta}d(x)\mathrm dV_g\right)^{\frac1q}\nonumber\\
&\leq\|u_n-u_0\|_{L^{\frac{p}{p-1}\frac{q}{q-1}}_{\sinh^\theta}(B_R^{\mathbb H})}^{\frac{p}{p-1}}\left(\int_{B_R^{\mathbb H}}e^{q\mu|u_n|^{\frac{p}{p-1}}}\sinh^\theta d(x)\mathrm dV_g\right)^{\frac1q}.\label{eqt10S}
\end{align}
By Proposition \ref{prop51}, the integral on the right-hand side is uniformly bounded in $n$. Using \eqref{eqtx} and \eqref{eqt10} with $u_n\to u_0$ in $L^{\frac{p}{p-1}\frac{q}{q-1}}_{\sinh^\theta}(B_R^{\mathbb H})$ we conclude
\begin{align*}
\mathcal{L}_{\mu,\mathbb H}&=\sup_{\substack{u\in W^{k,p}_{\mathbb H,\mathrm{rad}}(B_R^{\mathbb H}) \\ \|u\|_{W^{k,p}(B_R^{\mathbb H})}\leq1}}\int_{B_R^{\mathbb H}}e^{\mu|u|^{\frac{p}{p-1}}}\sinh^\theta d(x)\mathrm dV_g=\lim\int_{B_R^{\mathbb H}}e^{\mu|u_n|^{\frac{p}{p-1}}}\sinh^\theta d(x)\mathrm dV_g\\
&=\int_{B_R^{\mathbb H}}e^{\mu|u_0|^{\frac{p}{p-1}}}\sinh^\theta d(x)\mathrm dV_g.
\end{align*}

To prove that $\|u_0\|_{W^{k,p}(B_R^{\mathbb H})}=1$, assume $\|u_0\|_{W^{k,p}(B_R^{\mathbb H})}<1$ and define the function $\widetilde u=u_0/\|u_0\|_{W^{k,p}(B_R^{\mathbb H})}$ to get a contradiction with the fact that $u_0$ attains the supremum. We can assume that $u_0$ is nonnegative by replacing $u_0$ by $|u_0|$ if necessary. This concludes the proof.
\end{proof}

\begin{proof}[Proof of the Theorem \ref{theo15}]
$\mathrm{(a)}$ Note that, by item $\mathrm{(a)}$ of Theorem \ref{theo15},
\begin{equation*}
\int_{B_R^{\mathbb H}}|u|^{\frac{p}{p-1}}\sinh^\theta d(x)\mathrm dV_g=\omega_{N-1}\int_0^Re^{\mu|v|^{\frac{p}{p-1}}}\sinh^{\theta+N-1}(t)\mathrm dt<\infty,
\end{equation*}
for each $u\in W^{k,p}_{\mathbb H,\mathrm{rad}}(B_R^{\mathbb H})$.\\
$\mathrm{(b)}$ Let us assume that $u\in W^{k,p}_{\mathbb H,\mathrm{rad}}(B_R^{\mathbb H})$ with $\|u\|_{W^{k,p}(B_R^{\mathbb H})}\leq1$. Then, by \eqref{eqequiv}, $\|\omega^{\frac1p}_{N-1}v\|_{W^{k,p}_{\sinh^{N-1}}}\leq1$. Given $\mu<\mu_{0,\mathbb H}$, by the item $\mathrm{(b)}$ of Theorem \ref{theo15}, we have
\begin{equation*}
    \int_{B_R^{\mathbb H}}e^{\mu|u|^{\frac{p}{p-1}}}\sinh^{\theta}d(x)\mathrm dV_g=\omega_{N-1}\int_0^Re^{\mu\omega_{N-1}^{-\frac{1}{p-1}}|\omega_{N-1}^{\frac1p}v|^{\frac{p}{p-1}}}\sinh^{\theta+N-1}(t)\mathrm dt\leq C,
\end{equation*}
where $C>0$ does not depend on $u\in W^{k,p}_{\mathbb H,\mathrm{rad}}(B_R^{\mathbb H})$ with $\|u\|_{W^{k,p}(B_R^{\mathbb H})}\leq1$. The attainability was proved in Proposition \ref{propitemaS}.\\
$\mathrm{(c)}$ Given $\mu>C_0^{\frac{p}{p-1}}(\theta+N)[(k-1)!]^{\frac{p}{p-1}}$, by item $\mathrm{(c)}$ of Theorem \ref{theo16}, we obtain a sequence $(v_n)$ in $W^{k,p}((0,R),\sinh^{N-1}(t))$ such that $\|v_n\|_{W^{k,p}_{\sinh^{N-1}}}\leq1$ and
\begin{equation*}
\int_0^Re^{\mu C_0^{-\frac{p}{p-1}}|v_n|^{\frac{p}{p-1}}}\sinh^{\theta+N-1}(t)\mathrm dt\overset{n\to\infty}\longrightarrow\infty.
\end{equation*}
Setting $u_n(x)=C_0^{-1}v_n(d(x))$, and using item (3) of Theorem \ref{theo10}, we have $u_n\in W^{k,p}_{\mathbb H,\mathrm{rad}}(B_R^{\mathbb H})$ with $\|u_n\|_{W^{k,p}(B_R^{\mathbb H})}\leq1$ and
\begin{equation*}
\int_{B_R^\mathbb H}e^{\mu|u_n|^{\frac{p}{p-1}}}\sinh^{\theta}d(x)\mathrm dV_g=\omega_{N-1}\int_0^Re^{\mu C_0^{-\frac{p}{p-1}}|v_n|^{\frac{p}{p-1}}}\sinh^{\theta+N-1}(t)\mathrm dt\overset{n\to\infty}\longrightarrow\infty.
\end{equation*}
This completes the proof.
\end{proof}

\section{Final Comments}\label{sec8}

In this section, we present our concluding remarks on the limitations and potential applications of the results established in this paper. Specifically, we address the reasons why specific values of $\mu>0$ for which Theorems \ref{theo16} and \ref{theo15} are not applicable, providing a detailed discussion of these constraints. Additionally, we explore some scenarios in which our main theorems can be extended or applied to broader contexts, highlighting our work's versatility and potential impact.

\begin{remark}\label{remark}
The reader should note that Theorem \ref{theo15} does not guarantee whether $\mathcal{L}_{\mu,\mathbb H}$ is finite or infinite for $\mu$ within the interval $[\mu_{0,\mathbb H},\widetilde\mu_{0,\mathbb H}]$, where $\widetilde\mu_{0,\mathbb H}=C_0^{\frac{p}{p-1}}(\theta+N)[(k-1)!]^{\frac{p}{p-1}}$. Even by adapting the sequence used in Proposition \ref{prop54} for the space $W^{k,p}_{\mathbb H,\mathrm{rad}}(B_R^{\mathbb{H}})$, it is not possible to prove that $\mathcal{L}_{\mu,\mathbb H}=\infty$ for $\mu>\mu_{0,\mathbb H}$. Seeking insight from the sequences in \cite[Theorem 1.1]{nguyen2020sharp} and \cite[Theorem 1.1]{zbMATH06536855}, we attempted to proceed but could not conclude for $\mu>\mu_{0,\mathbb H}$.

Let us analyze the details of adjusting the sequence in Proposition \ref{prop54} for the space $W^{k,p}_{\mathbb H,\mathrm{rad}}(B_R^{\mathbb H})$. Define the sequence
\begin{equation*}
\widetilde\psi_{n,\varepsilon}(x)=H\left(\dfrac{\log\frac{R}{d(x)}}{\log n}\right),\quad x\in B_R^{\mathbb H}\backslash\{0\}.
\end{equation*}
Additionally, consider $\psi_{n,\varepsilon}\colon(0,R)\to\mathbb R$ such that $\widetilde\psi_{n,\varepsilon}(x)=\psi_{n,\varepsilon}(d(x))$ for each $x\in B_R^{\mathbb H}\backslash\{0\}$. As previously proved,
\begin{equation*}
    \|\psi_{n,\varepsilon}\|^p_{W^{k,p}_{\sinh^{N-1}}}\leq [(k-1)!]^p(\log n)^{1-p}\left[1+2^{p}\varepsilon\|\phi'\|^p_\infty+O\left((\log n)^{-1}\right)\right].
\end{equation*}
Thus, we expect that
\begin{equation}\label{eq999}
\|\widetilde\psi_{n,\varepsilon}\|^p_{W^{k,p}(B^{\mathbb H}_R)}\leq \omega_{N-1}[(k-1)!]^p(\log n)^{1-p}\left[1+2^{p}\varepsilon\|\phi'\|^p_\infty+O\left((\log n)^{-1}\right)\right].
\end{equation}
Assuming \eqref{eq999} and following the same arguments as in Proposition \ref{prop54}, we conclude that $\mathcal{L}_{\mu,\mathbb H}=\infty$ for $\mu>\mu_{0,\mathbb H}$. A straightforward calculation shows that
\begin{equation*}
\||\nabla^i\widetilde\psi_{n,\varepsilon}|_g\|^p_{L^p(B_R^{\mathbb H})}=O\left((\log n)^{-p}\right)\quad\forall i=0,\ldots,k-1.
\end{equation*}
However, the following inequality, which corresponds to the case $i=k$, does not hold:
\begin{equation*}
\||\nabla^k\widetilde\psi_{n,\varepsilon}|_g\|^p_{L^p(B_R^{\mathbb H})}\leq \omega_{N-1}[(k-1)!]^p(\log n)^{1-p}\left[1+2^{p}\varepsilon\|\phi'\|^p_\infty+O\left((\log n)^{-1}\right)\right]
\end{equation*}
This is because even if we isolate the higher order derivative, the other terms must be bounded by the norm of the higher order derivative. Note that
\begin{align*}
\left(\nabla^k\widetilde\psi_{n,\varepsilon}\right)_{i_1\ldots i_k}&=\left(\dfrac{2}{1-|x|^2}\right)^k\dfrac{x_{i_1}\cdots x_{i_k}}{|x|^k}\psi_{n,\varepsilon}^{(k)}(d(x))+\left(\mbox{terms with lower-order derivatives}\right)\\
&\leq\left(\dfrac{2}{1-|x|^2}\right)^k|\psi^{(k)}_{n,\varepsilon}(d(x))|^2+C\sum_{j=1}^{k-1}\dfrac{\psi^{(j)}(d(x))}{|x|^{k-j}}.
\end{align*}
By Proposition \ref{prophardy}, each $\|\frac{\psi^{(j)}(d(x))}{|x|^{k-j}}\|_{L^p_{\sinh^{N-1}}}$ can be estimated by $\|\psi^{(k)}\|_{L^{p}_{\sinh^{N-1}}}$, which is insufficient to establish the unboundedness of $\mathcal{L}_{\mu,\mathbb H}$ for $\mu>\mu_{0,\mathbb H}$. It is important to note that this issue also arises in the $\mathbb R^N$ scenario.
\end{remark}

\begin{center}\textbf{Critical case of the Adams-type inequality}\end{center} 

One may observe that Theorem \ref{theo16} does not provide any guarantees regarding $\mathcal{L}_{\mu}$ in the critical case $\mu = \mu_0$. To our knowledge, this problem remains unsolved, even in the power-type weight scenario; see \cite[Theorem 1.4]{arXiv:2302.02262}. This leads us to the following problem.
\begin{conj}
Let $\theta>-1$ and $W^{k,p}((0,R),\sinh^{N-1}(t))$ be the weighted Sobolev space with $0<R<\infty$, $N=kp$, and $p>1$. Then $\mathcal{L}_{\mu_0}$ is finite.
\end{conj}

\begin{center}\textbf{Application on a H\'enon-type equation}\end{center}

Using the embeddings established here, we consider the existence of solutions for a higher order class of elliptic problems, known as the higher order Hénon-type equation:
\begin{equation}\label{eqop1}
\left\{\begin{array}{cc}
     (-\Delta)^ku=\sinh^\theta d(x)|u|^{p-1}u,&\mbox{in }B_R^{\mathbb H},  \\
     \mathcal Bu=0,&\mbox{on }\partial B_R^{\mathbb H},
\end{array}\right.
\end{equation}
where $B_R^{\mathbb H}\subset\mathbb H^N$ is a bounded ball and $\mathcal Bu=0$ represents either the Navier or Dirichlet boundary conditions. The norm arising from $(-\Delta)^ku$ is equivalent to the one studied in this paper; see \cite{zbMATH03514476}. The study of \eqref{eqop1} remains an open problem in hyperbolic space. Under a subcritical condition $1<p<p^*$, the compact embedding from Theorem \ref{theo11} can be applied. For further details, when $k=2$ in $\mathbb R^N$, see \cite[Theorem 1.6]{zbMATH05978504} and \cite[Corollaire 2.1 and Th\'eor\'eme 1.1]{zbMATH04175546}. Moreover, for $k=1$ in hyperbolic space $\mathbb H^N$, see \cite{zbMATH06508496}.

\begin{center}\textbf{Application on Inequalities Involving Unbounded Domains}\end{center}

There are several approaches for studying Trudinger-Moser inequalities on unbounded domains, pioneered by the following classical works \cite{zbMATH01370657,zbMATH00062849,zbMATH01449917,zbMATH05270603}. One such approach involves examining exact growth inequalities for higher order derivatives. For comprehensive details on these exact growth inequalities, we refer to \cite{zbMATH06435874,zbMATH06497751,zbMATH06533141,zbMATH06950439} for $\mathbb{R}^N$ and \cite{zbMATH06579881,zbMATH07318489} for $\mathbb{H}^N$. Consider the following conjecture, which is expected at least under the assumption that $u$ is radial or without the weight $\sinh^\eta d(x)$. The following question corresponds to the weighted exact growth inequality established in \cite{zbMATH07827601}, adapted for the hyperbolic scenario.
\begin{conj}
Consider $\eta\in\mathbb R$ and $W^{k,p}_{\mathbb H}(\mathbb H^N)$ with $N=kp$. Assuming $p>1$ and some condition on $\eta$ and $\nu$, there exists $\beta_0>0$ such that for all $u\in W^{k,p}_{\mathbb H}(\mathbb H^N)$ satisfying $\|\nabla^ku\|_{L^p_{\nu}(\mathbb H^N)}\leq1$,
\begin{equation*}
\int_{\mathbb H^N}\dfrac{\exp_p(\beta_{0}|u|^{\frac{p}{p-1}})}{(1+|u|)^{\frac{p}{p-1}}}\sinh^\eta d(x)\mathrm dV_g\leq C\|u\|^p_{L^p_\eta(\mathbb H^N)}.
\end{equation*}
Moreover, there exists a sequence $(u_n)$ such that $\|\nabla^ku_n\|_{L^p_{\nu}(\mathbb H^N)}\leq 1$ and
\begin{equation*}
    \dfrac{1}{\|u_n\|^p_{L^p_{\eta}(\mathbb H^N)}}\int_{\mathbb H^N}\dfrac{\exp_p\left(\beta_{0}|u_n|^{\frac{p}{p-1}}\right)}{(1+|u_n|)^q}\sinh^\eta d(x)\mathrm dV_g\overset{n\to\infty}\longrightarrow \infty,\quad\forall q<\frac{p}{p-1}.
\end{equation*}
Furthermore, for all $\beta>\beta_{0}$ and $q\geq0$, we have
\begin{equation*}
\int_0^\infty\dfrac{\exp_p\left(\beta|u_n|^{\frac{p}{p-1}}\right)}{(1+|u_n|)^q}r^\eta\mathrm dr\overset{n\to\infty}\longrightarrow \infty.
\end{equation*}
\end{conj}

\begin{center}
\textbf{Application on Hardy-Sobolev-Maz'ya inequality}
\end{center}

We now compare our results on weighted embeddings for radial functions with existing Sobolev embeddings in the literature. We first recall that the GJMS operators on $\mathbb H^N$ is defined as follows (see \cite{zbMATH04198939},  \cite{zbMATH06197796}, and \cite{zbMATH06195229,zbMATH05987929})
 \begin{equation}\label{1.5}
 	P_{k}=P_{1}(P_{1}+2)\cdot\cdots\cdot(P_{1}+k(k-1)),\;\; k\in\mathbb{N},
 \end{equation}
 where $P_{1}=-\Delta_{\mathbb{H}}-\frac{n(n-2)}{4}$ is the conformal Laplacian on $\mathbb H^N$. The sharp Sobolev inequalities on $\mathbb H^N$ are given as follows (see \cite{zbMATH00945657} for the case $k=1$ and \cite{zbMATH00447008}, together with \cite{zbMATH06187289}, for the case $2\leq k<\frac{n}{2}$):
 \begin{equation}\label{1.6}
 	\int_{\mathbb H^N}(P_{k}u)u\mathrm dV_g\geq S_{n,k}\left(\int_{\mathbb H^N}|u|^{\frac{2n}{n-2k}}\mathrm dV_g\right)^{\frac{n-2k}{n}},\quad u\in
 	C^{\infty}_{0}(\mathbb H^N),\quad 1\leq k<\frac{n}{2},
 \end{equation}
 where $S_{n,k}$ is the best $k$-th order Sobolev constant.  
 
 Recently, Lu and Yang \cite{zbMATH07155018} have established higher order Poincare-Sobolev inequality on hyperbolic spaces (see also \cite{zbMATH07483893}), namely the so-called Hardy-Sobolev-Mazya inequality using Helgason-Fourier analysis on hyperbolic spaces (see \cite{zbMATH03864136}, \cite{zbMATH00615015}, \cite{zbMATH03916373}). 
 
 \begin{theorem}\label{th1.2}
 	Let $2\leq k<\frac{n}{2}$ and $2<p\leq\frac{2n}{n-2k}$. There exists a positive constant $C=C(n,p)$ such that for each $u\in C^{\infty}_{0}(\mathbb H^N)$,
 	\begin{equation}\label{1.11a}
 		\int_{\mathbb H^N}(P_{k}u)udV- \prod^{k}_{i=1}\frac{(2i-1)^{2}}{4}\int_{\mathbb H^N}u^{2}dV\geq C\left(\int_{\mathbb H^N}|u|^{p}dV\right)^{\frac{2}{p}}.
 	\end{equation}
 \end{theorem}
 The inequality \eqref{1.11a} for $k=1$ corresponds to the first-order Hardy-Sobolev-Maz'ya inequality, which was originally established by Maz'ya in half spaces (see \cite{zbMATH00195222} and \cite{zbMATH05604409}). Lu and Yang have also developed higher order versions of these Hardy-Sobolev-Maz'ya inequalities for complex hyperbolic spaces \cite{zbMATH07557162}, and by Flynn, Lu, and Yang for quaternionic and octonionic hyperbolic spaces \cite{zbMATH07835927}.

It remains unclear whether we can prove the inequality \eqref{1.11a} (or even \eqref{1.6}) for symmetric functions without the any boundary conditions on bounded geodesic balls $B_R^{\mathbb H}\subset\mathbb H^N$. Specifically, it is unknown whether results can be established for symmetric functions on $B^{\mathbb H}_R$ without the any boundary conditions that improves the inequality in Theorem \ref{theo11} by subtracting a Hardy term. Additionally, it would be interesting to explore whether a weighted version of the inequality \eqref{1.11a} can be established for radial functions on $\mathbb H^N$. These remain intriguing questions for future research.

\bigskip 

\begin{flushleft}
 {\bf Funding:}  
 J. M. do \'O acknowledges partial support from CNPq through grants 312340/2021-4, 409764/2023-0, 443594/2023-6, CAPES MATH AMSUD grant 88887.878894/2023-00 and Para\'iba State Research Foundation (FAPESQ), grant no 3034/2021. G. Lu acknowledges partial support from Simons collaboration grants 519099 and 957892 from the Simons Foundation. R. C. Ponciano acknowledges partial support from CAPES through grants 88887.633572/2021-00, 88881.689999/2022-01, and São Paulo Research Foundation (FAPESP) grant 2023/07697-9.\\
\end{flushleft}

\bibliographystyle{abbrv}
\bibliography{bib}

\end{document}